\documentclass[12pt]{article}
\usepackage[backend=bibtex,maxbibnames=99,sorting=nyt,style=numeric,firstinits=true, url=false,doi=false,isbn=false]{biblatex} 
\renewbibmacro{in:}{
  \ifboolexpr{%
     test {\ifentrytype{article}}%
     }{}
  {\printtext{{I}n\intitlepunct}}%
}  

\renewbibmacro*{publisher+location+date}{
  \printlist{publisher}%
    {\setunit*{\addcomma\space}}
    {\setunit*{\addcomma\space}}%
  \iflistundef{location}
    {\setunit*{\addcomma\space}}
    {\setunit*{\addcomma\space}}%
  \printlist{location}%
  \setunit*{\addcomma\space}%
  \usebibmacro{date}%
  \newunit}

\DeclareFieldFormat[book]{volume}{vol. #1 \addcomma \space} 
\DeclareFieldFormat[article]{volume}{#1 \addcomma \space} 
\DeclareFieldFormat{journaltitle}{#1 \space} 
\DeclareFieldFormat{booktitle}{#1 \addcomma \space} 
\DeclareDelimFormat{finalnamedelim}{\addcomma \space}  
\DeclareFieldFormat[article,book,inbook,incollection,inproceedings,thesis,misc]{title}{\textit{#1} \addcomma \space} 
\DeclareDelimFormat{nametitledelim}{\addcomma\space}
\DeclareFieldFormat{edition}{#1 \addcomma \space} 
\DeclareFieldFormat[article,incollection]{pages}{#1} 

\usepackage[colorlinks=true,citecolor=blue]{hyperref}
\usepackage{mathptmx, amsthm, mathptmx, enumerate, color}
\usepackage{epstopdf}

\usepackage{graphicx}
\usepackage{mathtools,amsmath,amssymb,amsfonts,bbm, bm, dsfont, mathrsfs}
\usepackage[all]{xy}
\usepackage{centernot}

\usepackage{enumitem} 
\renewcommand\labelenumi{(\roman{enumi})}
\renewcommand\theenumi\labelenumi

\usepackage{caption,tikz,pst-all,pst-solides3d,tikz,wrapfig}
\usetikzlibrary{hobby}

\usepackage{graphicx}
\usepackage{pgfplots}
\usepgfplotslibrary{fillbetween}

\pgfplotsset{
  flaeche/.style={draw=none,fill=gray,fill opacity=0.2}
}

\newtheorem{theorem}{Theorem}[section]
\newtheorem{corollary}[theorem]{Corollary}
\newtheorem{lemma}[theorem]{Lemma}

\theoremstyle{definition}
\newtheorem{definition}[theorem]{Definition}
\newtheorem{example}[theorem]{Example}
\newtheorem{remark}[theorem]{Remark}

\newtheorem{assumption}{Assumption}

\numberwithin{equation}{section}

\let\emptyset\varnothing

\newcommand{\D}{\mathcal{D}}
\newcommand{\R}{\mathbb{R}}
\newcommand{\C}{\mathcal{C}}
\newcommand{\N}{\mathbb{N}}

\newcommand{\A}{\mathcal{A}}

\newcommand{\B}{\mathcal{B}}

\newcommand{\F}{\mathcal{F}}
\newcommand{\X}{\mathcal{X}}

\newcommand{\U}{\mathcal{U}}

\newcommand{\M}{\mathcal{M}}
\renewcommand{\L}{\mathcal{L}}
\renewcommand{\P}{\mathcal{P}}
\renewcommand{\S}{\mathcal{S}}

\renewcommand{\1}{\mathds{1}}

\newcommand{\abs}[1]{\left| #1\right|}
\newcommand{\norm}[1]{\left\| #1\right\|}

\newcommand{\Span}{\mathrm{span}\hspace{0.06cm}} 
\newcommand{\Int}{\mathrm{int}\hspace{0.06cm}}
\newcommand{\dom}{\mathrm{dom}\hspace{0.06cm}}
\newcommand{\lev}{\mathrm{lev}\hspace{0.06cm}}
\newcommand{\epi}{\mathrm{epi}\hspace{0.06cm}}
\newcommand{\bd}{\mathrm{bd}\hspace{0.06cm}}
\newcommand{\cl}{\mathrm{cl}\hspace{0.06cm}}
\newcommand{\rec}{\mathrm{rec}\hspace{0.06cm}}

\newcommand{\prob}{\mathbb{P}}
\newcommand{\var}{\mathrm{VaR}\hspace{0.06cm}}

\newcommand{\avar}{\mathrm{AVaR}\hspace{0.06cm}}

\usepackage[hang]{footmisc}

\title{A new view on risk measures associated with acceptance sets}

\author{Marcel Marohn\footnote{Martin-Luther-University Halle-Wittenberg, Faculty of Natural Sciences II, Institute of Mathematics, 06099 Halle (Saale), Germany (E-mail: \href{mailto:marcel.marohn@b-tu.de}{marcel.marohn@b-tu.de})} 
   \and Christiane Tammer\footnote{Martin-Luther-University Halle-Wittenberg, Faculty of Natural Sciences II, Institute of Mathematics, 06099 Halle (Saale), Germany (E-mail:  \href{mailto:christiane.tammer@mathematik.uni-halle.de}{christiane.tammer@mathematik.uni-halle.de}
)}
}

\begin{document}
\fontfamily{cmss}\selectfont
%

%
%
%

\maketitle


\vskip 8mm {\noindent {\bf Abstract.}
In this paper, we study properties of certain risk measures associated with acceptance sets. These sets describe regulatory preconditions that have to be fulfilled by financial institutions to pass a given acceptance test. If the financial position of an institution is not acceptable, the decision maker has to raise new capital and invest it into a basket of so called eligible assets to change the current position such that the resulting one corresponds with an element of the acceptance set. Risk measures have been widely studied, see e.g. \cite{Fischer.2018} for an overview. The risk measure that is considered here determines the minimal costs of making a financial position acceptable. In the literature, monetary risk measures are often defined as translation invariant functions and, thus, there is an equivalent formulation as Gerstewitz-Functional (see Artzner et al. \cite{Artzner.1999} and Jaschke, Küchler \cite{Jaschke.2001}). The Gerstewitz-Functional is an useful tool for separation and scalarization in multiobjective optimization in the non-convex case. In our paper, we study properties of the sublevel sets, strict sublevel sets and level lines of a risk measure defined on a linear space. Furthermore, we discuss the finiteness of the risk measure and relax the closedness assumptions. \\

\section{Introduction}

Scalarization and separation of sets are important topics for many fields of research in mathematics and mathematical economics like multiobjective  optimization, risk theory, optimization under uncertainty and financial mathematics. For a given real vector space $\X$ over $\R$, nonlinear translation invariant functionals $\varphi_{\A,K} : \X \rightarrow \mathbb{R} \cup \{-\infty\} \cup \{+ \infty \}$ are considered with 
\begin{equation}\label{funcak00}
\varphi_{\A,K} (X):= \inf \{t\in
{\mathbb{R}} \mid X\in tK + \A\}
\end{equation}
where $\emptyset\neq \A\subseteq \X$ and $K\in\X\backslash\{0\}$ such that $\A-\R_+ K \subseteq \A$. Functionals given by \eqref{funcak00} are used by Gerstewitz in \cite{Gerstewitz.1984} for deriving separation theorems for not necessarily convex sets as separating functional and as scalarization functional of vector optimization problems. Formulations of $\varphi_{\A,K}$ convenient in risk theory are subject of this paper. Relationships between coherent risk measures and the functional \eqref{funcak00} are studied by Artzner et al. \cite{Artzner.1999} and Jaschke, K\"uchler \cite{Jaschke.2001} (see also \cite{Heyde.2006}, \cite{Hamel.2015}, \cite{Scandolo.2004}, \cite{Frittelli.2006} and references therein). For a given acceptance set $\A$ in the space of financial positions $\X$, we consider the functional
\begin{align}\label{rho_baes}
\X\ni X \mapsto \rho_{\A,\M,\pi}(X):=\inf \{\pi(Z)\mid Z\in \M, X+Z\in \A\}
\end{align} 
following Farkas et al. \cite{Farkas.2015} and Baes et al. \cite{Baes.2020}, where $\pi\colon \M \rightarrow \R$ is a pricing functional defined on a set of allowed movements $\M \subseteq \X$ by which an investor can change the current financial position $X$. The functional \eqref{rho_baes} is a risk measure describing the costs for satisfying regulatory preconditions. It is a generalization of \eqref{funcak00} through the simultaneous allowance of a set of directions instead of one fixed direction $K\in \X$ and the valuation of the movements by a general functional $\pi$. However, in the so called Reduction Lemma in \cite{Farkas.2015}, Farkas et al. showed that $\rho_{\A,\M,\pi}$ can be reduced to a functional from type \eqref{funcak00}. It is shown that this class of functionals from \eqref{funcak00} coincides with the class of translation invariant functions and, thus, such functionals are employed as coherent risk measures in risk theory (see \cite{Artzner.1999}), since translation invariance is a basic property of risk measures.  

\bigskip

Considering risk measures with respect to acceptance sets is from special interest in financial mathematics. Since the financial crisis revealed many deficits in risk taking and management of financial institutions, regulation acquired intensively greater importance in recent years, leading to restrictions for business activities like minimum capital requirements for credit risk of banks, see e.g. Basel III preconditions in \cite{BCBS.2010} and \cite{BCBS.2011}.
That leads to partly massive restrictions of possible investment stategies and to reduced gains by a risk-return trade-off, see e.g. \cite{Bali.2006}. Thus, finding optimal portfolios under given regulatory acceptance conditions is essential for the survive of an institution. In \cite{Marohn.2020}, we followed Baes et al. in \cite{Baes.2020} and considered a portfolio optimization problem with the aim of fulfill regulatory preconditions (we also say ``reaching acceptability'' because they are described by an acceptance set $\A$) under minimal costs for closed acceptance sets.

\bigskip

The paper is organized as follows: After a short overview about the needed mathematical background, we describe the basic financial model, which we consider in Sections 3 and 4. This includes the vector space of financial positions $\X$, the set of eligible payoffs $\M\subseteq \X$ and the linear pricing functional $\pi\colon \M \rightarrow \R$. Here, we give a definition for acceptance sets $\A\subseteq \X$ by Baes et al. in \cite{Baes.2020}, but we do not additionally assume closedness. Moreover, we state some very important examples from a practical point of view. Afterwards, we study properties of the risk measure $\rho_{\A,\M,\pi}$ in Section 5. 


\section{Preliminaries}
In this article, certain standard notions and terminology are used. Let $\X$ be a \textit{vector space over $\R$}. The extended set of real numbers is denoted by $\overline{\R} :=\R\cup\{-\infty,+\infty\}$, the set of non-negative real numbers by $\R_+$ , the set of positive real numbers by $\R_>$ and the set of non-positive real numbers by $\R_-$. Consider subsets $\A, B \subseteq \X$. Then,
\begin{align*}
\A+\B:=\{X+Y\mid X\in \A, Y\in \B\}
\end{align*}
is the \textit{Minkowksi sum} of $\A$ and $\B$. For simplicity, we replace $\{X\}+\A$ by $X+\A$ for the sum of a set $\A$ and a single element $X\in \X$. We set
\begin{align*}
\lambda \A = \{\lambda X \mid X\in \A\} \text{ with } \lambda \in \R,
\end{align*}
which leads for the Minkowski sum of $\A$ and $-\B$ to
\begin{align*}
\A-\B=\{X-Y\mid X\in \A, Y\in \B\}.
\end{align*} 
The cardinality of $\A$ is denoted by $\abs{\A}$. $\A$ is \textit{star-shaped (around $0$)} if 
\begin{align*}
\forall X\in \A, \forall \lambda \in [0,1]:\quad \lambda X \in \A.
\end{align*}
and $\A$ is  \textit{convex} if 
\begin{align*}
\forall X,Y\in \A, \forall \lambda \in [0,1]:\quad \lambda X + (1-\lambda)Y \in \A.
\end{align*}
It is obvious that $\A$ is star-shaped if it is convex and $0\in \A$. A nonempty set $\A$ is called a \textit{cone} if
\begin{align*}
\forall \ X\in \A, \forall \lambda \geq 0:\ \lambda X \in \A.
\end{align*}
Let $\C\subseteq \X$ be a cone. Then, $\C$ is star-shaped. Furthermore, $\C$ is said to be \textit{proper} or \textit{nontrivial} if $\{0\}\subsetneq \C \neq \X$. We call $\C$ \textit{pointed} if $\C\cap (-\C)=\{0\}$ and \textit{reproducing} if $\C-\C=\X$. For each cone $\C$ holds 
\begin{align*}
\C \text{ convex } \quad \Longleftrightarrow \quad \C + \C \subseteq \C.
\end{align*}
A cone from special interest is the \textit{recession cone} of a subset $\A\subset \X$, i.e.,
\begin{align*}
\rec \A := \{X\in \X\mid Y+\lambda X \in \A\ \text{ for all } Y\in \A, \lambda \in \R_+\}. 
\end{align*}

\begin{definition}[see Guti{\'e}rrez et al. \cite{Gutierrez.2016}]\label{d-direct_cl}
Let $\X$ be a vector space over $\R$, $\A\subseteq \X$ and $K\in \X\backslash\{0\}$. The \textit{$K$-directional closure of $\A$} is given by
\begin{align*}
\cl_K (\A):=\{X\in \X\mid \forall \lambda \in \R_>\ \exists t\in \R_+\text{ with }t<\lambda \text{ and }X-tK\in \A\}.
\end{align*}
$\A$ is said to be \textit{$K$-directionally closed} if $\A=\cl_K (\A)$. The \textit{$K$-directional interior of $\A$} is given by
\begin{align*}
\Int_K(\A):= \{X\in \X \mid \exists \lambda \in \R_> \ \forall t\in [0,\lambda]: X+tK\in \A\} 
\end{align*}
and the \textit{$K$-directional boundary of $\A$} is 
\begin{align*}
\bd_K(\A):=\cl_K(\A)\backslash \Int_K(\A).  
\end{align*}
\end{definition}

 The following properties will be useful and are collected from \cite{Tammer.2020}: 

\begin{lemma}[see Tammer and Weidner {\cite[Lemma 2.3.24]{Tammer.2020}}]\label{lem:k-directionally_cl_sequences}
Let $\X$ be a vector space over $\R$, $\A\subseteq \X$ and $K\in \X\backslash \{0\}$. Then,
\begin{align*}
\cl_K (\A)=\{X\in \X\mid \exists (t_n)_{n\in \N}\subseteq \R_+: t_n \downarrow 0\ \forall n\in \N: X-t_nK \in \A\}.
\end{align*}
\end{lemma}

\begin{lemma}[see Tammer and Weidner {\cite[Lemma 2.3.26 and 2.3.42, Prop. 2.3.48 and 2.3.49]{Tammer.2020}}]\label{lem:k-directionally_properties}
Let $\X$ be a vector space over $\R$, $\A\subseteq \X$ and $K\in \X\backslash \{0\}$. Then, the following holds:
\begin{enumerate}[font = \normalfont]
\item $\A\subseteq \cl_K (\A) \subseteq \A+\R_+ K$,
\item $\cl_K(\cl_K(\A))=\cl_K(\A)$,
\item $\cl_K (\A) - \R_+ K = \cl_K(\A -\R_+ K)=\cl_K(\A-\R_> K)$,
\item $\A-\R_> K = \cl_K(\A)-\R_> K = \Int_K(\A - \R_> K)=\Int_K(\A-\R_+ K )$.
\end{enumerate}
\end{lemma}

\begin{lemma}[see Tammer and Weidner {\cite[Prop. 2.3.29 and 2.3.53]{Tammer.2020}}]\label{lem:k-directionally_properties_K_in_minusrec}
Let $\X$ be a vector space over $\R$, $\A\subseteq \X$ and $K\in \X\backslash \{0\}$. Suppose $K\in -\rec(\A)$. Then, the following holds: 
\begin{enumerate}[font = \normalfont]
\item $\cl_K(\A)=\{X\in \X \mid X-\R_> K \subseteq \A\}$,
\item $\Int_K(\A)=\{X\in \A\mid \exists t\in \R_>: X+tK\in \A\}=\A-\R_>K$, 
\item $\bd_K(\A)=\{X\in \X\mid \forall t\in \R_>: X+tK\notin \A \text{ and }X-tK\in \A\}$. 
\end{enumerate}
\end{lemma}

\bigskip

Now, we suppose that $(\X,\tau)$ is a \textit{topological vector space over $\R$}. Note that we just write $\X$ for $(\X,\tau)$ if $\tau$ is obvious or not from interest. Elements of $\tau$ are called \textit{open} sets and elements of $\X\backslash \tau$ are called \textit{closed}. If $\{0\}$ is closed, $\X$ is called \textit{Hausdorff}. We say that $\X$ is \textit{locally convex} if it exists $\{\U_i\}_{i\in I}\subseteq \tau$ of $0$ such that the sets $\U_i$ are convex. For $\A\subseteq \X$, $\cl (\A)$ denotes  the \textit{closure}, $\bd (\A)$ the \textit{boundary} and $\Int (\A)$ the \textit{interior} of $\A$. The following lemma gives a connection between topological and directional properties of $\A$:

\begin{lemma}[see Tammer and Weidner {\cite[Prop. 2.3.54 and 2.3.55]{Tammer.2020}}]\label{lem:k-directionally_topol_properties}
Let $\X$ be a topological vector space over $\R$, $\A\subseteq \X$ and $K\in \X\backslash \{0\}$. Then, the following holds:
\begin{enumerate}[font = \normalfont]
\item $\cl_K(\A)\subseteq \cl(\A)$,
\item $\Int(\A)\subseteq \Int_K(\A)\subseteq \A$,
\item $\bd_K(\A)\subseteq \bd (\A)$.
\end{enumerate}
\end{lemma}

We call $(X,\norm{ \cdot })$ a \textit{normed vector space} where $\X$ is a vector space (here assumed to be) over $\R$ and $\norm{ \cdot }$ denotes a \textit{norm} on $\X$, i.e., a map $\norm{ \cdot } \colon \X\rightarrow \R_+$ such that for every $X,Y\in \X$ and $\alpha\in \R$ the following holds: (i) $\norm{X}=0 \Rightarrow X = 0$, (ii) $\norm{\alpha X} =\abs{\alpha} \norm{X}$ and (iii) the triangle inequality $\norm{X+Y}\leq \norm{X} + \norm{Y}$. We write just $\X$ for $(\X,\norm{ \cdot })$ if the norm is obvious or not from interest. A sequence $(X_k)_{k\in \N}\subseteq \X$ is said to \textit{converge} in $\X$ to\ the \textit{limit} $X\in \X$ (and we write $X_k\rightarrow X$ for $k\rightarrow +\infty$) if $\lim\limits_{k\rightarrow +\infty} \norm{X_k-X}=0$. We say $(\X,\norm{ \cdot })$ is \textit{complete} if each Cauchy sequence has a limit in $\X$. A complete normed vector space is called \textit{Banach space}. The set 
\begin{align*}
\B_r(Z):=\left\{W\in \X\mid \norm{W-Z}\leq r\right\}
\end{align*}
denotes the closed ball with center $Z\in \X$ and radius $r>0$.

\begin{example}\label{expl:Lp space} 
There are many possible choices for $\X$ in the modeling process. Since we will deal with random variables describing payoffs of assets or portfolios and financial risk measures, there are some widely used and suitable spaces, see e.g. Liebrich and Svindland \cite{Liebrich.2017}. For measure spaces $\X=(\Omega,\F,\mu)$, we consider probability measures $\prob\colon \F\rightarrow [0,1]$ on a sigma-algebra $\F\subseteq \P(\Omega)$ with state space $\Omega$ instead of general measures $\mu\colon \F\rightarrow \overline{\R}$. One possibility for the choice of $\X$ is the \textit{space of bounded random variables} 
\begin{align*}
\B(\Omega,\F,\prob)=\{X\colon \Omega \rightarrow \R \mid \abs{X(\omega)}<+\infty \text{ for all }\omega \in \Omega\}. 
\end{align*}
If we consider the norm $\norm{X}_\B=\sup\limits_{\omega \in \Omega} \abs{X(\omega)}$, then ($\B(\Omega,\F,\prob),\norm{ \cdot }_\B$) is a Banach space. Moreover, one can consider $\L^p$-spaces: If we identify $X\equiv Y$ if and only if $\prob (X\neq Y) =0$, then the $\L^p$-space consists for $p\in[1,+\infty)$ of real-valued $p$-integrable functions on a measure space $(\Omega,\F,\mu)$. Hence, the \textit{space of $p$-integrable random variables} is denoted by $\L^p(\Omega,\F,\prob)$ for $1\leq p <+\infty$ and is given by 
\begin{align*}
\L^p(\Omega,\F,\prob) = \{X\colon \Omega \rightarrow \R| \norm{X}_{\L^p}< +\infty\}
\end{align*} 
with $\norm{X}_{\L^p}=(\mathbb{E}(\abs{X}^p))^{\frac{1}{p}}=\left(\int_\Omega \abs{X}^p \ d\prob \right)^{\frac{1}{p}}$. To extend the definition to the case $p=+\infty$, we set
\begin{align*}
\L^\infty(\Omega,\F,\prob)=\{X\colon \Omega \rightarrow \R|\norm{X}_{\L^\infty} < +\infty\}
\end{align*}
with $\norm{X}_{\L^\infty}=\inf\{C\geq 0| \abs{X(\omega)} \leq C\ \prob \text{-a.s.}\}$ and call $\L^\infty(\Omega,\F,\prob)$ the \textit{space of essential bounded random variables}. As before, if the parameters $\Omega$, $\F$ and $\prob$ do not matter or are clear, we just write $\L^p$. The spaces $(\L^p,\norm{ \cdot }_{\L^p})$ and $(\L^\infty,\norm{ \cdot }_{\L^\infty})$ are Banach spaces, since almost-sure identical random variables are identified. Furthermore, $\L^q\subseteq \L^p$ for $1\leq p \leq q\leq +\infty$ because of $\prob(\Omega)<+\infty$.
\end{example}

\bigskip
Consider a map $f \colon \X\rightarrow \overline{\R}$ where $\X$ is a vector space  over $\R$. Then,
\begin{align*}
\mathrm{dom} f:=\{X\in \X\mid f(X)< +\infty\}
\end{align*}
is the \textit{domain} of $f$ and 
\begin{align*}
\mathrm{epi} f:=\{(X,t)\in \X\times \R\mid  f(X)\leq t\}
\end{align*}
the \textit{epigraph} of $f$. We call $f\colon \X\rightarrow \overline{\R}$ \textit{proper} if $\mathrm{dom} f \neq \emptyset$ and $f(X)>-\infty$ for all $X\in \X$. If $\mathrm{epi} f$ is convex, then $f$ is said to be \textit{convex}. $f$ is said to be \textit{linear} if 
\begin{align*}
\forall X,Y\in \X, \forall \lambda, \mu \in \R: \quad f(\lambda X + \mu Y) = \lambda f(X) + \mu f(Y).
\end{align*}
The \textit{kernel} or \textit{null space} of a linear map $f$ is the subspace
\begin{align*}
\ker f := \{X\in \X\mid f(X)=0\}.
\end{align*}
We call
\begin{align*}
\lev_{f,=}(\alpha):=\{X\in \X \mid f(X)=\alpha\}
\end{align*}
\textit{level line},
\begin{align*}
\lev_{f,\leq}(\alpha):=\{X\in \X \mid f(X)\leq\alpha\}
\end{align*}
\textit{sublevel set} and 
\begin{align*}
\quad \lev_{f,<}(\alpha):=\{X\in \X \mid f(X)<\alpha\}
\end{align*}
\textit{strict sublevel set} of $f$ to the level $\alpha\in \R$. For short, we denote each of the sets $\lev_{f,=}(\alpha)$, $\lev_{f,\leq}(\alpha)$ and $\lev_{f,<}(\alpha)$ \textit{level set}. \\ 

Now, we suppose that $\X$ is a topological vector space  over $\R$. $f$ is called \textit{lower semicontinuous} if $\mathrm{epi} f$ is closed. The set
\begin{align*}
\X^* := \{\varphi\colon \X \rightarrow \R\mid  \varphi \text{ linear and continuous}\}
\end{align*}
denotes the \textit{topological dual space}. If $\X$ is a locally convex Hausdorff space with $\X\neq \{0\}$, there is some non-trivial linear continuous functional, i.e., $\X^*\neq \{0\}$.  Equivalently,
\begin{align*}
\forall X\in \X\backslash \{0\} \ \exists \varphi \in \X^*:\qquad \varphi(X)> 0.
\end{align*}
Consequently, for given $X,Y\in \X$ with $X\neq Y$, there exists $\varphi \in \X^*$ with $\varphi(X)\neq \varphi(Y)$.  \\

In vector optimization and other applications it is necessary to compare elements $X,Y\in \X$. Hence, let $\mathcal{R}\subseteq \X\times \X$ be a binary relation on $\X$ and we set $X\mathcal{R}Y$ equivalently for $(X,Y)\in \mathcal{R}$, compare \cite{Gopfert.2003} for standard terminology and examples. For our purposes we remember that $\mathcal{R}$ is a \textit{partial order} if it is reflexive, antisymmetric and transitive. $\X$ is said to be \textit{partially ordered by $\mathcal{R}$} if $\mathcal{R}$ is a partial order on $\X$. In the following, we write $\leq$ for $\mathcal{R}$. 

Cones in $\X$ are suitable for describing binary relations on $\X$ and, especially, partial orders. The following theorem specifies that:

\begin{theorem}[see \cite{Gopfert.2003}, Theorem 2.1.13]\label{thm:partial_order_by_cones}
Let $\X$ be a vector space  over $\R$, $\leq$ a reflexive binary relation on $\X$ that is compatible with the linear structure of $\X$ and $\C\subseteq \X$ a cone. 
\begin{enumerate}[font = \normalfont]
\item Consider 
\begin{align}\label{def:partial order relation R_C}
\leq_\C := \{(X,Y)\in \X\times \X\mid  Y-X\in \C\} .
\end{align}
Then, $\leq_\C$ is a binary relation on $\X$, too, and fulfills the following properties:
\begin{enumerate}[font = \normalfont]
\item $\leq_\C$ is reflexive and compatible with the linear structure of $\X$,
\item $\leq_\C$ is transitive $\Leftrightarrow$ $\C$ is convex,
\item $\leq_\C$ is antisymmetric $\Leftrightarrow$ $\C$ is pointed.
\end{enumerate}
\item Consider
\begin{align*}
\C_\leq:=\{X\in \X\mid 0 \leq X\}.
\end{align*}
Then, $\C_\leq$ is a cone such that for all $X,Y\in \X$: 
\begin{align*}
X \leq Y \quad \Longleftrightarrow \quad X \leq_{(\C_\leq)} Y.
\end{align*}
\end{enumerate}
\end{theorem}

By Theorem \ref{thm:partial_order_by_cones}, $\leq_\C$ from \eqref{def:partial order relation R_C} is a partial order if and only if $\C\subseteq \X$ is a convex, pointed cone. Thus, we call a  convex, pointed cone $\C$ an \textit{ordering cone}. The corresponding partial order $\leq_\C$ is given by 
\begin{align}\label{def:partial order with ordering cone}
\forall X,Y\in \X:\ X\leq_\C Y :\Longleftrightarrow Y-X\in \C.
\end{align}
For example, if $\X$ is partially ordered by $\leq$, the natural ordering cone in $\X$ is the \textit{positive cone}
\begin{align}\label{eq:pos_cone_def}
\X_+=\{X\in \X\mid  0 \leq X\}
\end{align}
and the corresponding partial order is simplified denoted $\leq$ instead $\leq_{\X_+}$. Every element $X\in \X$ is called \textit{positive}. One example for $\leq$ is the component-wise ordering on $\X=\R^n$. 


\section{The financial market} \label{sec:finmodel} 
Our basic market model refers to Baes et al. \cite{Baes.2020} and Farkas, Koch-Medina, Munari \cite{Farkas.2015}. We consider an one-period-model of a financial market, where an investor choices his or her portfolio at the time $t=0$, which results in a capital position with some (in general) random payoff at the future time $t=1$. In this section, we present this model in more detail.

\bigskip

Throughout this paper, we consider a real vector space $\X$ called the \textit{space of capital positions}, usually a space of random variables. Furthermore, let $\Omega$ be the set of possible states in $t=1$, $\F$ be a $\sigma$-Algebra on $\Omega$ and $\mathbb P\colon \F \rightarrow [0,1]$ a probability measure. 

\bigskip

While Baes et al. \cite{Baes.2020} assume a locally convex Hausdorff topological vector space over $\R$ fulfilling the first axiom of countability and Farkas et al. \cite{Farkas.2015} assume a topological vector space  over $\R$, we suppose a real vector space $\X$, which we only extend to be equipped with a topology and further properties where necessary. Sometimes we speak of financial positions instead of capital positions $X\in \X$. Nevertheless, $X$ is the capital of an investor  in the future time $t=1$ and is given by the residuum of assets and liabilities, i.e., positive outcomes are gains and negative outcomes are losses. $\X$ is typically chosen as a space like $\L^p$ with $1\leq p < +\infty$, see Example \ref{expl:Lp space}. In some situations, we consider a finite set $\Omega=(\omega_1,\dots,\omega_n)$, i.e., vectors $X\in\R^n$ with $X_k=X(\omega_k)$ ($k=1,\dots,n$). For each event $\B\in \F$, we set
\begin{align*}
X\in \B \ \mathbb{P}-\text{a.s.} \quad := \quad \mathbb{P}(X\in \B)=\mathbb{P}(\{\omega \in \Omega \mid X(\omega)\in \B\})=1,
\end{align*}
where ''a.s.'' means ''almost surely''. If $X,Y\in \X$ are random variables, we write $X=Y$ and $X\neq Y$ if and only if $\mathbb{P}(X=Y)=1$ and $\mathbb{P}(X=Y)=0$, respectively. Note that $\X$ could be any normed vector space. Furthermore, we suppose $\X$ to be partially ordered by the pointed convex cone $\X_+$. The cone is represented by the order relation $\leq$ as given in \eqref{eq:pos_cone_def}.
If we consider a space of random variables, e.g. $\X=\L^p(\Omega,\F,\prob)$, we understand $0 \leq X$ in the sense of $\prob$-a.s. 

In the following the superscript $^T$ always denotes transposed vectors. As mentioned in the beginning, an investor can invest into a finite set $\S$ of  \textit{eligible assets} 
\begin{align}\label{def:eligible_assets}
S^i=(S_0^i,S_1^i)^T,\quad  i\in\{0,1,\dots,N\} 
\end{align}
with $N\in\N$ in time $t=0$. Here, $S_0^i\in \R$ denotes the price in $t=0$ for one unit of the liquid asset and $S_1^i\in \X$ denotes the (in general random) payoff in $t=1$ for each unit. We set 
\begin{align}\label{def:eligible_assets_secure}
S^0:=(1,(1+r)\1_\Omega)^T=(1,1+r)^T.
\end{align}
$S^0$ describes a secure investment opportunity with interest rate $r\in \R_+$. Secure means that the payoff is a constant. For every constant random variable $X=c\cdot \1_\Omega$, i.e., $\prob(X=c)=1$, with $c\in \R$ arbitrary, we just write $X=c$. For a collection of all prices or payoffs, respectively, we set
\begin{align}\label{def:S_j}
S_j:=(S_j^0,S_j^1,\dots,S_j^N)^T, \qquad j\in \{0,1\}.
\end{align}
\begin{remark}
The secure asset $S^0$ is not directly assumed by Baes et al. in \cite{Baes.2020} or Farkas et al. in \cite{Farkas.2015}, but makes sense for economical reasons. The existence of such an asset is typical for the most common models in modern financial economics like the Capital Asset Pricing Model (CAPM), which is one of the most known and used equilibrium asset pricing models in theory and practice, see Sharpe \cite{Sharpe.1964}, Lintner \cite{Lintner.1965} and Mossin \cite{Mossin.1966}. The CAPM outlines the relationship between risk and expected return of an asset. Like originally shown by Sharpe in \cite{Sharpe.1963}, the firm-specific risks can be diversified, but not the systematic risk (market risk), which is an assumption that seems to be verified regarding the recent Corona-crisis, as seen for example in March 2020 when all global stock markets fell around 30 \% - 40 \% in regard to the beginning of the year due to the pandemic. Consequently, only the systemic risk is rewarded with a risk premium, which the CAPM and generalizations of it also retain, see e.g. \cite{Bodie.2018}.

 The secure investment opportunity could be a central bank account or a U.S. treasury bond. We suppose for convenience no interest payments, i.e., $r=0$, which, of course, is a major simplification as well as that there is only one secure alternative which is especially independent from the time horizon. Moreover, it is well known that $r\geq 0$ is not always true in practice. For instance, European banks are penalized by $r=-0.4 \%$, called deposit facility, since March 2016 if they park money at the European Central Bank (ECB) instead of invest or lend it \cite{EuropeanCentralBank.2019}. Thus, the bank pays money for lending it the ECB. This phenomen does not only concern banks: it is also challenging for assurances. These have to invest money into theoretical secure assets like government bonds by law, but that bonds partly have negative effective interest rates, too, see for example the yields of German government bonds on the 2nd August 2019, which were negative for every maturity.
\end{remark}

There has been many research concerning the case of one eligible asset (also, how to choose it), see e.g. Farkas et al. for a defaultable bond $\X=\B(\Omega)$ in \cite{Farkas.2014} and \cite{Farkas.2014b}. Other spaces have been studied, too, see e.g. Kaina and Rüschendorf in \cite{Kaina.2009} for $\X=\L^p$ or Cheridito and Li in \cite{Cheridito.2009} for Orlicz spaces. The actions a decision maker can take into account are described by suitable portfolios of the assets $S^i$. The subspace
\begin{equation}\label{calM}
\M:=\Span(1,S_1^1,\dots,S_1^N)
\end{equation}
of $\X$ spanned by the secure payoff and $S_1^i$ with $i=1,\dots,N$ is called \textit{space of the eligible payoffs}. We assume linear independent eligible payoffs. Consequently,
\begin{align}\label{eq:dim_M}
1 < \dim (\M) <  +\infty.
\end{align}
Every $Z\in \M$ describes the random payoff related to a portfolio of those eligible assets. To ensure $S_1^0=1\in \M \subseteq \X$ for cases like $\X=\L^p(\Omega,\F,\mu)$, we have to assume a finite measure space, i.e., $\mu(\Omega)< +\infty$, which is guaranteed here by a probability measure $\mu=\prob$, i.e., $\prob(\Omega)=1$ for arbitrary $\Omega$. If $\X$ is a topological vector space, then $\M$ is equipped with the relative topology induced by $\X$ which is normable through $\dim \M <  +\infty$. In the following, the fixed norm in $\M$ for topological vector spaces $\X$ is given by $\norm{ \cdot }$.


As we consider a financial market, we make the following typical assumptions:

\begin{assumption}\label{Ass1}
For $\M$ being the subspace of $\X$ given by \eqref{calM}, the \textit{Law of One Price} holds, i.e., for all $Z\in \M$ there exists $c\in \R$ such that
\begin{align}
\forall x\in \R^{N+1} \text{ with }Z  = S_1^T x: \quad S_0^T x = c \label{LawOnePrice}
\end{align} 
with $S_j$ given by \eqref{def:S_j} for $j\in\{0,1\}$. Furthermore, the \textit{no-arbitrage principle} is fulfilled, i.e., there is no arbitrage opportunity (in the sense of Irle \cite[Def. 1.10]{Irle.2012}, see Remark \ref{rem:arbitrage}). Especially,  for all $x\in \R^{N+1}$ with payoff $Z=S_1^Tx$, it holds that
\begin{align}
(S_0^T x \leq 0\ \land \ Z\geq 0\ \prob - a.s.) \quad \Longrightarrow \quad S_0^Tx = 0 = \prob(Z > 0).
\end{align}
\end{assumption}

\begin{remark}\label{rem:arbitrage}
The no-arbitrage principle in Assumption \ref{Ass1} uses the arbitrage terminology from Irle \cite[Def. 1.10]{Irle.2012}, where an arbitrage opportunity is defined as $x\in \R^{N+1}$ with 
\begin{align*}
S^T_0 x \leq 0\ \land \ \prob(S_1^T x\geq 0)=1 \quad \text{ and it holds } \quad S_0^Tx < 0 \ \lor \  \prob(S_1^Tx>0)>0.
\end{align*}
Irle notes in \cite[Anmerkung 1.11]{Irle.2012} that an one-period-model is arbitrage-free if there is no arbitrage opportunity with $S_0^T x = 0$. In the definition of arbitrage by Föllmer and Schied in \cite[Def. 1.2]{Follmer.2016}, the case $S_0^Tx<0$ is not included. 

Nevertheless, we consider two types of arbitrage here, namely \textit{free lunch}, i.e., $S_0^Tx<0$, and \textit{free lottery} (or  \textit{money machine}), i.e., $\prob(S_1^T x >0)>0$ (see Bamberg \cite{Bamberg.2003}). An illustrative introduction to the arbitrage principle and mathematical studies for especially derivative assets can be found in \cite{Varian.1987} and a bride more general mathematical introduction in \cite{Delbaen.2008}. In \cite{Hull.2015b} and \cite{Bodie.2018}, the economical background is presented, especially the distinction between arbitrage, hedging and speculation, and its role in capital market theory. The arbitrage in the explicit example of Japan is studied in \cite{Miyazaki.2007}.
\end{remark}

Portfolios $x=(x_0, x_1,\dots, x_N)^T\in \R^{N+1}$ with the same payoff have the same initial price by Law of One Price, which, especially, leads to a unique price for every eligible payoff $Z\in \M$. Following Baes et al. \cite{Baes.2020}, we define a \textit{pricing functional} $\pi\colon \M\rightarrow \R$ (with $\M$ given by \eqref{calM}) as
\begin{align}
\pi(Z):=S_0^Tx  \ \text{ for all }x\in\R^{N+1}:\ Z=S_1^T x.\label{def:pi_Z}
\end{align}
Of course, $\pi$ is a linear operator. If $\X$ is a topological vector space, then $\pi$ is also continuous, since $\dim \M <+\infty$. An important property is the monotonicity of $\pi$. In contrast to \cite{Marohn.2020}, where the absence of good deals was required, we show that $\pi$ is always monotonically increasing, i.e., 
\begin{align}\label{eq:mon_increasing}
\forall Z_1,Z_2 \in \M:\quad Z_2-Z_1\in \X_+ \quad \Longrightarrow \quad \pi(Z_1)\leq \pi(Z_2).
\end{align}
 
\begin{lemma}\label{lem:monotonicity_of_pi}
Let Assumption \ref{Ass1} be fulfilled. Then, the pricing functional $\pi\colon \M\rightarrow \R$ from \eqref{def:pi_Z} is monotonically increasing on $\M$ (see \eqref{eq:mon_increasing}).
\end{lemma}

\begin{proof}
Let $Z_1,Z_2\in \M$ with $Z_2-Z_1\in \X_+$ and $Z_1\neq Z_2$ $\prob$-a.s., i.e., 
\begin{align*}
\prob(Z_2-Z_1>0)=1.
\end{align*} 
If $\pi(Z_2)<\pi(Z_1)$, i.e., $\pi(Z_2-Z_1)<0$ holds,
then $Z_2-Z_1\in \M$ is a free lunch - arbitrage as mentioned in Remark \ref{rem:arbitrage}, which contradicts the no-arbitrage-principle in Assumption \ref{Ass1}. Thus, $\pi(Z_2)\geq \pi(Z_1)$ must hold.
\end{proof}

\begin{remark}\label{rmk:Xplus_pi_pos}
The Law of One Price in Assumption 1 secures that $\pi$ is well-defined while the no-arbitrage principle implies monotonicity of $\pi$, as seen in the proof of Lemma \ref{lem:monotonicity_of_pi}. Especially, if there exist no arbitrage opportunities, then $\ker \pi \cap \X_+ =\{0\}$ and, thus, we have $\pi(Z)>0$ for all $Z\in (\M\cap\X_+)\backslash \{0\}$ by monotonicity of $\pi$. Conversely, monotonicity of $\pi$ does not imply that the no-arbitrage-principle is fulfilled, e.g. if $\X=\R^2=\M$ and $\pi(Z)=Z_2$, since $\ker \pi \cap \X_+ = \R_+\times \{0\}$.
\end{remark}

All eligible assets with the same price $m\in \R$ are summarized by
\begin{align}
\pi_m:=\{Z\in \M\mid  \pi(Z)=m\}. \label{def:pi_m}
\end{align} 
If there are $Z_1,Z_2\in \M$ with $Z_2-Z_1\in \X_+$ and $Z_1\neq Z_2$, then $\M \cap \X_+ \neq \{0\}$. Especially, we can consider $Z_1=0$. As in \cite{Baes.2020}, we assume the existence of $U\in \M\cap \X_+$ with strict positive price, i.e.,
\begin{align*}
\X_+ \cap \bigcup_{m>0} \pi_m \neq \emptyset.
\end{align*}
Since $\pi$ is a linear functional, we can simplified assume $\pi(U)=1$.

\begin{assumption}\label{ann:u_payoff}
For $\M$ being the subspace of $\X$ given by (\ref{calM}) and $\pi$ defined by \eqref{def:pi_Z}, there exists some positive payoff $U\in \M\cap \X_+$ with $\pi(U)=1$.
\end{assumption}

Note that the Law of One Price from Assumption \ref{Ass1} is automatically fulfilled by Assumption \ref{ann:u_payoff} because the existence of $\pi$ is required in Assumption \ref{ann:u_payoff}.

\begin{remark}\label{rem:constants_in_M}
Since $1\in \M$, we can set $U=S_1^0 =1$. In general, $m =m\cdot S^0_1\in \M$ for constant $m\in \R$ because $\Span S_1^0=\R\subseteq \M$. Indeed, since $S^0=(1,1)$, we have 
\begin{align*}
\pi(m)=m\cdot \pi(S_1^0)=m\cdot S_0^0 = m \cdot 1 =m.
\end{align*}
Note that we will work with arbitrary $U\in \M\cap \X_+$ in the following instead of fixing $U=1$ for application of our results to models in other research where $1\in \M$ is not assumed.
\end{remark}

As observed in \cite{Farkas.2015} for a topological vector space $\X$ and proved in \cite{Marohn.2020}, we can rewrite \eqref{def:pi_m} by Assumption \ref{ann:u_payoff} for $\X$ being a vector space, since no topological properties are required in the proof:

\begin{lemma}[see \cite{Farkas.2015}]\label{lem:pi_m_rewrite_U}
Let Assumption \ref{ann:u_payoff} be fulfilled. Accordingly to Assumption \ref{ann:u_payoff}, we consider an arbitrarily chosen element $U\in \M\cap \X_+$ with $\pi(U)=1$ and $m\in \R$. Then,
\begin{align*}
\pi_m= m\cdot U + \ker \pi
\end{align*}
for $\pi_m$ given by \eqref{def:pi_m}.
\end{lemma}

\begin{remark}
It can be shown for $\pi(U)\in \R\backslash\{0\}$ arbitrary that $\pi_m$ coincides with $\frac{m}{\pi(U)}\cdot U + \ker \pi $ for every $m\in \R$, see \cite{Farkas.2015}. 
\end{remark}

\begin{remark}\label{rem:ker_pi_neq_0}
The kernel of $\pi$ will be very important for our results. By Rank-Nullity Theorem,
\begin{align*}
\dim(\M)=\dim(\ker \pi) + \dim (\mathrm{Im}(\pi))
\end{align*} 
holds and, thus, $\dim(\ker \pi)=\dim (\M)- 1$, since $\mathrm{Im}(\pi)=\R$ for $\pi \not\equiv 0$ (this is excluded by existence of $U$ in Assumption \ref{ann:u_payoff}). Consequently, $\ker \pi \neq \{0\}$ because $\dim(\M)>1$ is supposed by \eqref{eq:dim_M}. 
\end{remark}

\section{Risk associated with regulatory constraints}

Let $\X$ be a vector space above $\R$, $\M\subseteq \X$ the space of eligible payoffs $S^i$ from \eqref{calM} fulfilling \eqref{eq:dim_M} and $\pi\colon \M\rightarrow \R$ a pricing functional given by \eqref{def:pi_Z}. The so called acceptance set specifies those positions which are allowed to be occupied. The decision maker has to decide which actions can be undertaken to modify the current financial position such that the new one is acceptable if the current one is not. On the other hand, if the current position is already acceptable, the decision maker could set money free while the position is remaining acceptable and the money could be used otherwise. A suitable set for describing all capital positions $X\in \X$ being acceptable capitalized with respect to regulatory constraints can be defined in the following way (see Baes et al. \cite{Baes.2020} and Artzner et al. \cite{Artzner.1999}):

\begin{definition}\label{def:acceptanceset}
Let $\X$ be a vector space over $\R$. A set $\A \subseteq \X$ is called \textit{acceptance set} if the following conditions hold:
\begin{enumerate}
\item $0\in \A$,
\item $\A \subsetneq \X$ (proper),
\item $\A+\X_+\subseteq \A$, i.e., for all $X\in \A$, $Y \in X + \X_+$ implies $Y\in \A$ (monotonicity).
\end{enumerate}
\end{definition}

Now, after we completed our financial setting, we consider the financial model (FM) in the space of capital positions $\X$ with the probability space ($\Omega, \F, \mathbb{P}$) (see Section \ref{sec:finmodel}) throughout the paper:\\ 

\fbox{\parbox{\linewidth}{
\begin{align*}
(\mathrm{FM}):\quad &S^i\in \R\times \X\ (i=0,1,\dots,N)\text{ is the }i\text{-th eligible asset from }\eqref{def:eligible_assets}\\
& \text{with }S^0 \text{ being the secure investment opportunity from \eqref{def:eligible_assets_secure}},\\
&\M \text{ is the subspace of }\X \text{ given by } \eqref{calM} \text{ fulfilling }\eqref{eq:dim_M},\\
&\pi\colon \M\rightarrow \R \text{ is the pricing functional given by }\eqref{def:pi_Z},\\
&\A\subseteq \X \text{ is an acceptance set according to Definition \ref{def:acceptanceset}}\\
&\text{and Assumption \ref{Ass1} is fulfilled.}
\end{align*}}}\\

\begin{remark}\label{rem:acceptance_sets}
Acceptance sets are nonempty by (i). Note that (i) and (iii) in Definition \ref{def:acceptanceset} provide $\X_+\subseteq \A$ and, especially, $m\cdot \1_\Omega \in \A$ for all $m\in \R_+$. For $U\in \M\cap \X_+$ as in Assumption \ref{ann:u_payoff}, we obtain by (iii)
\begin{align*}
\A + \R_+ U \subseteq \A,
\end{align*}
i.e., $U\in \rec \A$. More precisely, we have 
\begin{align}\label{eq:A_plus_mU_in_A}
\A + \R_+ U = \A.
\end{align}
Furthermore, we get $U\in \A$ by (i) and \eqref{eq:A_plus_mU_in_A}.
\end{remark}

The definition is motivated by Baes et al. \cite{Baes.2020}, but we do not assume $\A$ to be closed, in general. 

\begin{remark}
All properties of the risk measure that we are studying in Section \ref{sec:prop_risk_measure} can be shown without any assumption with respect to the closedness of $\A$. In \cite{Farkas.2015}, only $\emptyset\neq \A\subsetneq \X$ and $\A+\X_+\subseteq \A$ are required by an acceptance set (which is called a capital adequacy test there) and $(\A,\M,\pi)$ is called \textit{risk measurement regime}. Farkas et al. argue that their stated properties can be united with expectations from nontrivial capital adequacy tests. That means, especially, that positions are automatically acceptable if they dominate any acceptable position. Baes et al. observe in \cite{Baes.2020} that their requirements with additionally assumption $0\in \A$ and closedness of $\A$ are widely assumed in practice.  By $0\in \A$, it is acceptable if a financial position is constant zero. This property can be easily reached for other acceptance sets by translation. Note that convexity of $\A$ is often required in other frameworks, but some essential acceptance sets do not fulfill this, see the following example. Further terminology was introduced by several authors, like \textit{convex acceptance sets} by F\"ollmer and Schied in \cite{Follmer.2002} and Frittelli and Rosazza Gianin in \cite{Frittelli.2002} for $\A$ being convex, \textit{conic acceptance set} by Farkas et al. in \cite{Farkas.2015} for $\A$ being a cone and \textit{coherent acceptance sets} by Artzner et al. in \cite{Artzner.1999} for $\A$ being a convex cone.
\end{remark}

Acceptance sets are mostly given by risk measures in practice. An axiomatic approach to define (coherent) risk measures is introduced by Artzner et al. in \cite{Artzner.1999} and generalized to convex risk measures by Föllmer and Schied in \cite{Follmer.2002} and Frittelli and Rosazza Gianin in \cite{Frittelli.2002}: 

\begin{definition}[see \cite{Artzner.1999},\cite{Follmer.2002}, \cite{Frittelli.2002}]\label{def:risk_measure}
Let $\X$ be a real vector space. A functional $\rho\colon \X\rightarrow \overline{\R}$ is called \textit{(monetary) risk measure} if the following conditions hold:
\begin{enumerate}
\item $\forall X,Y\in \X:\ Y-X\in \X_+ \Rightarrow \rho(Y)\leq \rho(X)$ (monotonicity),
\item $\forall m\in \R, \forall X\in \X:\ \rho(X+m)=\rho(X)-m$ (translation invariance).  
\end{enumerate} 
A (monetary) risk measure $\rho$ is called \textit{convex risk measure} if it fulfills additionally
\begin{align}\label{eq:rho_convex}
\forall X,Y\in \X, \forall \lambda \in [0,1]:\ \rho(\lambda X + (1-\lambda)Y) \leq \lambda \rho(X) + (1-\lambda)\rho(Y).
\end{align}
$\rho$ is called \textit{coherent risk measure} if it is a convex risk measure fulfilling the following property of positive homogeneity:
\begin{align}\label{eq:rho_pos_hom}
\forall X \in \X, \forall \lambda \in \R_+:\quad \rho(\lambda X) = \lambda \rho(X).
\end{align}
\end{definition}

\begin{remark}\label{rem:def_riskmeasure}
In Definition \ref{def:risk_measure}, the property translation invariance is also known as \textit{cash invariance}. It is important to mention that we consider functionals from $\X$ into the space of extended real values $\overline \R = \R \cup \{+\infty\} \cup \{-\infty\}$ with the properties (i) and (ii) in Definition \ref{def:risk_measure}. Furthermore, there are different definitions of risk measures in the literature. While some authors define a risk measure as any map $\rho\colon \X\rightarrow \R$ (see Artzner et al. \cite[Def. 2.1]{Artzner.1999}), i.e., an arbitrary map with overall finiteness, other authors define risk measures as a map $\rho\colon \X\rightarrow \R$ with the properties in Definition \ref{def:risk_measure} (see Föllmer, Schied \cite[Def. 4.1]{Follmer.2016}) or as a map $\rho\colon \X\rightarrow \R\cup\{+\infty\}$ with the properties in Definition \ref{def:risk_measure}, but  assume $\rho(0)\in \R$ (see Föllmer, Schied \cite[Def. 2.1]{Follmer.2010}) additionally to the required properties in our more general definition. Note that $\rho(0)=0$ for positive homogeneous risk measures (see \eqref{eq:rho_pos_hom}), since
\begin{align*}
\forall \lambda \in \R_+:\quad \rho(0)=\rho(\lambda \cdot 0) = \lambda \rho(0)
\end{align*}
holds.
\end{remark}

Convex risk measures are important because only these take diversification into account. More exactly, diversification means that the decision maker invests a portion $\lambda \in [0,1]$ into a possible strategy or investment opportunity with output $X\in \X$ and the remaining part into another one with output $Y\in \X$. The convexity of the risk measure (see \eqref{eq:rho_convex}) implies that diversification should not increase the risk, which can be expressed by the (for monetary risk measures as given by Definition \ref{def:risk_measure}) equivalent condition of \textit{quasi-convexity} (see \cite{Farkas.2018}), i.e., 
\begin{align*}
\forall X,Y\in \X, \forall \lambda \in [0,1]:\quad \rho(\lambda X + (1-\lambda) Y) \leq \max\{\rho(X),\rho(Y)\}.
\end{align*}
As mentioned in Example \ref{expl:Lp space}, typical spaces of capital positions are $\L^p$-space, especially with $1\leq p \leq +\infty$, see e.g. \cite{Frittelli.2002} and \cite{Follmer.2002}. Biagini and Frittelli showed in \cite{Biagini.2009} that there are no finite convex risk measures for $\X=\L^p$ with $0\leq p < 1$ which are not constant. 

\begin{example}[see F\"ollmer, Schied \cite{Follmer.2016} and Baes et al. \cite{Baes.2020}]\label{expl:VaR_CVaR_acc_sets} Let $(\Omega,\F,\prob)$ be a probability space. Consider some confidence level $\alpha \in (0,1)$. Furthermore, let $\rho\colon \X\rightarrow \overline{\R}$ be a (monetary) risk measure with $\rho(0)\leq 0$. Then,
\begin{align*}
\A_\rho:=\{X\in \X\mid \rho(X)\leq 0\}
\end{align*}
is an acceptance set. $\A_\rho$ is closed if $\rho$ is a coherent risk measure because in that case $\rho$ is continuous and $\A_\rho$ is the pre-image of $\R_-$, see Artzner et al. \cite{Artzner.1999}, Prop. 2.2. Furthermore, $\rho$ is also continuous and, thus, $\A_\rho$ closed  if $(\X,\norm{ \cdot })$ is a Banach lattice, e.g. $\L^p$ with $\norm{ \cdot }_{\L^p}$ or $\L^\infty$ with $\norm{ \cdot }_{\L^\infty}$, and $\rho$ is finite, i.e., $\rho\colon \X \rightarrow \R$, see \cite{Liebrich.2017}. $\A_\rho$ is convex (a cone) if and only if $\rho$ is convex (positively homogeneous), see \cite{Follmer.2016}, Prop. 4.6. Thus, $\rho$ is a coherent risk measure if and only if $\A_\rho$ is a convex cone.

In the following, $X\in \X$ describes a random payoff with a return distribution. An example for a risk measure is the \textit{Value-at-Risk of $X\in \X$ at the level $\alpha$}, which is given by
\begin{align*}
\var_\alpha(X):=\inf\{m\in \R\mid  \prob(X+m<0)\leq \alpha\},
\end{align*}
see \cite[Def. 4.45]{Follmer.2016}. It can be interpreted as the smallest amount of capital that has to be added to the financial position $X$ to reach a probability of a loss that is not higher than $\alpha$. The corresponding acceptance set $\A_{\var_\alpha}$ is a cone, since $\var_\alpha$ is a positively homogeneous (monetary) risk measure. Therefore, $\A_{\var_\alpha}$ is a conic acceptance set (see \cite{Follmer.2016}), which does not have to be convex, since $\var_\alpha$ is not convex.

Another example for $\rho$ is the \textit{Average-Value-at-Risk of $X\in \X$ at the level $\alpha$} (also known as \textit{Conditional-Value-at-Risk} or \textit{Expected Shortfall}, see \cite{Follmer.2016}), which is given by
\begin{align}\label{def:cvar}
\avar_\alpha(X):=\frac{1}{\alpha}\int_0^\alpha \var_s(X)\ ds,
\end{align}
see \cite[Def. 4.48]{Follmer.2016} and \cite{Pflug.2000}. The corresponding acceptance set $\A_{\avar_\alpha}$ is a closed, convex cone and, therefore,  a coherent acceptance set, since $\avar_\alpha$ is a positively homogeneous, quasi-convex (monetary) risk measure (see \cite[Theorem 4.52]{Follmer.2016}).  Note that, as mentioned before Example \ref{expl:VaR_CVaR_acc_sets}, quasi-convex risk measures are also convex. Consequently, quasi-convex risk measures are coherent risk measures if they are positive homogeneous.
 
 In general, institution do not have to fulfill just one regulatory precondition. If the single preconditions are described by acceptance sets $\A_j$, $j=1,\dots, m$, then, obviously, 
 \begin{align*}
\A_{reg}:=\bigcap_{j=1}^m \A_j
\end{align*} 
is an acceptance set, too. 
\end{example}

A typical assumption in financial mathematics is that it can not be found a \textit{good deal}:

\begin{assumption}[absence of good deals] \label{ann:absence_good_deals}
Consider (FM). It holds $\pi(Z)>0$ for all payoffs ${Z \in (\A\cap \M)\backslash \{0\}}$.
\end{assumption}

The following lemma gives a characterization for the absence of good deals. Baes et al. derived the following result in \cite[Prop. 2.6 (iii)]{Baes.2020} for closed acceptance sets in a locally convex Hausdorff topological vector space above $\R$ fulfilling the first axiom of countability. However, it is possible to show a corresponding result for our setting, as well.

\begin{lemma}\label{lem:gooddeal}
Consider \emph{(FM)}. Let Assumption \ref{ann:u_payoff} be fulfilled and
\begin{align}\label{eq:span_U_no_good_deals}
\A\cap (-\R_> U) = \emptyset
\end{align}
with $U\in \M\cap \X_+$ being the payoff according to Assumption \ref{ann:u_payoff}. Then, the following conditions are equivalent:
\begin{enumerate}[font = \normalfont]
\item $\nexists Z \in (\A\cap \M)\backslash \{0\}: \  \pi(Z)\leq 0$,
\item $\A\cap \ker \pi = \{0\}$.
\end{enumerate}
\end{lemma}

An example that the equivalence in Lemma \ref{lem:gooddeal} does not hold if \eqref{eq:span_U_no_good_deals} is not fulfilled can be found in \cite[Example 2.9]{Baes.2020}. 

\begin{remark}
In Lemma \ref{lem:gooddeal}, (i) $\Rightarrow$ (ii) holds even if \eqref{eq:span_U_no_good_deals} is not fulfilled. Sufficient conditions for $\eqref{eq:span_U_no_good_deals}$  are also observed by Baes et al. in \cite{Baes.2020}, Prop. 2.6, namely $\A\cap(-\A)=\{0\}$ and $\A\cap (-\X_+)=\{0\}$. Note that, although $\A$ is assumed to be closed in \cite{Baes.2020}, the closedness is not used in the proof there. Consequently, the result holds in our setting, as well.
\end{remark}

To pass an acceptability test given by $\A$, the decision maker is obviously interested in the minimal capital amount that has to be raised and invested into the assets $S^i$, $i=0,1,\dots,N$ from \eqref{def:eligible_assets} to reach a new capital position $X^0\in \A$. Hence, following \cite{Farkas.2015} and \cite{Baes.2020}, we consider the nonlinear functional $\rho_{\A,\M,\pi}\colon \X \rightarrow \overline{\R}$ with arbitrary subset $\A\subseteq \X$ and subspace $\M\subseteq \X$ given by
\begin{align}
\rho_{\A,\M,\pi}(X):=\inf \{\pi(Z)\mid Z\in \M, X+Z\in \A\}.\label{def:rho}
\end{align}
Functionals from type \eqref{def:rho} were studied and called \textit{capital requirement} in \cite{Scandolo.2004} and \cite{Frittelli.2006}. In a similiar financial setting as here with $\A$ being an acceptance set and $\M$ defined as in \eqref{calM}, the functional \eqref{def:rho} was studied by Farkas et al. in \cite{Farkas.2015} and by Baes et al. in \cite{Baes.2020}, where $\rho_{\A,\M,\pi}(X)$ describes the capital requirement for the situation we mentioned at the beginning. We call $\rho_{\A,\M,\pi}$ the \textit{risk measure on $\X$ associated to $\A$ and $\M$}. Note that there is no probability measure necessary for defining $\rho_{\A,\M,\pi}$. In the case of $\rho_{\A,\M,\pi}(X)<0$, the current capital position $X$ can be changed under setting money free to reach acceptability which does not mean that $X$ is already acceptable (even if $\rho_{\A,\M,\pi}(X)=0$), i.e., 
\begin{align*}
\rho_{\A,\M,\pi}(X)\leq 0 \quad \centernot\Longrightarrow \quad X\in \A,
\end{align*}
see \cite[Remark 4.4]{Marohn.2020}.
On the other hand,
\begin{align*}
X\in \A \quad \Longrightarrow \quad \rho_{\A,\M,\pi}(X)\leq 0
\end{align*}
because $0\in \M$ and $\pi(0)=0$.

\section{Properties of the risk measure $\rho_{\A,\M,\pi}$} \label{sec:prop_risk_measure}

In this section, we study $\rho_{\A,\M,\pi}$  given by \eqref{def:rho} with respect to (FM) in more detail. 
 First, we recall a result by Farkas et al. \cite{Farkas.2015} (a proof can be found in \cite{Marohn.2020}) that $\rho_{\A,\M,\pi}$ is monotone and translation invariant, i.e., a risk measure, indeed (see Definition \ref{def:risk_measure}). Although Farkas et al. \cite{Farkas.2015} worked in a topological vector space, the result also works for vector spaces, since no topological properties are used in the proof.
 
\begin{lemma}[see Farkas et al. {\cite[Lemma 2.8]{Farkas.2015}}]\label{lem:rho_riskmeasure}
Consider \emph{(FM)}. Let Assumption \ref{ann:u_payoff} be fulfilled and $\rho_{\A,\M,\pi}\colon \X \rightarrow \overline{\R}$ be the functional given by \eqref{def:rho}. Then, the following holds:
\begin{enumerate}[font = \normalfont]
\item $\rho_{\A,\M,\pi}(X)\geq \rho_{\A,\M,\pi}(Y)$ for all $X,Y\in \X$ with $Y\in X + \X_+$,
\item $\rho_{\A,\M,\pi}(X+Z)=\rho_{\A,\M,\pi}(X)-\pi(Z)$ for all $X\in \X, Z\in \M$,
\end{enumerate}
\end{lemma}

\begin{remark}\label{rem:rho_normalized}
Property (i) means that $\rho_{\A,\M,\pi}$ is monotone while (ii) is also known as \textit{cash invariance or translation invariance} and has important consequences. Let $X\in \X$ arbitrary and $U\in \M\cap\X_+$ given as in Assumption \ref{ann:u_payoff}. By (ii), we have especially
\begin{align*}
\rho_{\A,\M,\pi}(X+mU)=\rho_{\A,\M,\pi}(X)-m
\end{align*}
for all $m\in \R$ and, thus,
\begin{align*}
\rho_{\A,\M,\pi}(X+\rho_{\A,\M,\pi}(X)U)=0.
\end{align*}
If $\A$ fulfills \eqref{eq:span_U_no_good_deals}, see Lemma \ref{lem:gooddeal}, and Assumption \ref{ann:absence_good_deals} is fulfilled, then $\rho_{\A,\M,\pi}$ is \textit{normalized}, i.e., $\rho_{\A,\M,\pi}(0)=0$, which implies $0\in \bd_{-U} (\A)$ and, thus, $0\in \bd (\A)$ for topological vector spaces $\X$. Consequently, 
\begin{align*}
\forall Z \in \M:\ \rho_{\A,\M,\pi}(Z)=\rho_{\A,\M,\pi}(0)-\pi(Z) = -\pi(Z).
\end{align*}
Thus, (ii) is equivalent to \textit{cash additivity}, i.e.,
\begin{align}\label{eq:rho_cashadd}
\forall X\in \X\ \forall Z\in \M: \quad \rho_{\A,\M,\pi}(X+Z)=\rho_{\A,\M,\pi}(X)+\rho_{\A,\M,\pi}(Z).
\end{align} 
\end{remark}

By Lemma \ref{lem:rho_riskmeasure}, the functional $\rho_{\A,\M,\pi}$ from \eqref{def:rho} is a (monetary) risk measure, see Definition \ref{def:risk_measure}. Note that we do not assume $\rho(0)\in \R$ in contrast to some other definitions for risk measures in the literature, see Remark \ref{rem:def_riskmeasure}. Especially, $\rho_{\A,\M,\pi}$ is no coherent risk measure (see Definition \ref{def:risk_measure}) in general, since $\rho_{\A,\M,\pi}$ is not always normalized, i.e., $\rho_{\A,\M,\pi}(0)\neq 0$, and, thus, $\rho_{\A,\M,\pi}$ is not positive homogeneous (see Remark \ref{rem:def_riskmeasure}). Moreover, $\rho_{\A,\M,\pi}$ is not always convex. Nevertheless, $\rho_{\A,\M,\pi}$ is a coherent risk measure if $\A$ is a convex cone (see Lemma \ref{lem:further_properties_of_rho}).

$\rho_{\A,\M,\pi}$ is a translation invariant functional by Lemma \ref{lem:rho_riskmeasure}(ii). As mentioned in the introduction, $\rho_{\A,\M,\pi}$ can be seen as a generalization of the scalarizing functional (\ref{funcak00}), which has monotonicity and translation invariance properties, too.  More exactly, $\rho_{\A,\M,\pi}$ can be reduced to a functional of that type, which is given by the following important relation between $\rho_{\A,\M,\pi}$ and the payoff $U\in\M\cap \X_+$ according to Assumption \ref{ann:u_payoff}, the so called \textit{Reduction Lemma}:

\begin{lemma}[see Farkas et al. \cite{Farkas.2015}, Reduction Lemma 2.10]\label{lem:reductionlemma}
Consider \emph{(FM)}. Let Assumption \ref{ann:u_payoff} be fulfilled by $U\in \M\cap \X_+$ and $\rho_{\A,\M,\pi}\colon \X \rightarrow \overline{\R}$ be the functional given by \eqref{def:rho}. Then,
\begin{align*}
\forall X\in \X: \quad \rho_{\A,\M,\pi}(X)=\inf\{m\in \R\mid X+mU \in \A+\ker\pi\}.
\end{align*}
\end{lemma}

The proof of Lemma \ref{lem:reductionlemma} in \cite{Farkas.2015} does not require any topological properties: it only uses the representation 
\begin{align*}
\rho_{\A,\M,\pi}(X)=\inf\{m\in \R\mid (X + \pi_m)\cap \A \neq \emptyset \}
\end{align*}
and Lemma \ref{lem:pi_m_rewrite_U}. Hence, we formulated Lemma \ref{lem:reductionlemma} for vector spaces instead of topological vector spaces as in \cite{Farkas.2015}.

\begin{remark}\label{bem:rho_finite_optz_empty}
Note that $\rho_{\A,\M,\pi}(X)\in \R$ does not necessary lead to the existence of a movement $Z\in \M$ such that $X+Z\in \A$ with $\pi(Z)=\rho_{\A,\M,\pi}(X)$, see e.g. \cite{Marohn.2020}, Example 4.15. By Lemma \ref{lem:reductionlemma}, minimal costs for reaching acceptability can be determined by just considering movements $mU\in \M$ with $m\in \R$ and $U\in \M\cap \X_+$ the payoff according to Assumption \ref{ann:u_payoff} instead of all movements $Z\in \M$. Therefore, $\rho_{\A,\M,\pi}(X)$ is given by the minimal $m\in \R$ such that the movement through $mU$ results in a position in $\A+\ker \pi$. Consequently, $\rho_{\A,\M,\pi}$ can be reduced to a functional from the type \eqref{funcak00} through
\begin{align*}
\rho_{\A,\M,\pi}(X)=\inf\{m\in \R\mid X\in \A+\ker\pi-mU\}=\varphi_{\A+\ker \pi,-U}(X).
\end{align*} 
\end{remark}

The set $\A+\ker \pi$ will be important for our studies and can be naturally interpreted as the set of those capital positions that can be made acceptable by zero costs. $\A+\ker \pi$ is nonempty, since $0\in \A+\ker \pi$. Properties of $\A+\ker \pi$ with respect to the direction $U$ will be from special interest. The following lemma states that $\A+\ker \pi$ fulfilles the monotonicity property for acceptance sets:

\begin{lemma}\label{lem:monotonicity_of_A_ker_pi}
Consider \emph{(FM)}. Let $X\in \A+\ker \pi$. Then, $Y\in \A+\ker \pi$ for every $Y\in \X$ with $Y-X\in \X_+$.
\end{lemma}

\begin{proof}
Let $X,Y\in \X$ with $X\in \A+\ker \pi$ and $Y-X\in \X_+$. Then, there exist $X^0\in \A$ and $Z^0\in \ker \pi$ with $X^0=X-Z^0$. Consequently, $Y-Z^0\in \A$ by monotonicity of $\A$, see Definition \ref{def:acceptanceset}(iii), since 
\begin{align*}
(Y-Z^0)-X^0 =  (Y-Z^0)-(X-Z^0) = Y-X \in \X_+.
\end{align*}
Thus, $Y=(Y-Z^0)+Z^0\in \A+\ker \pi$ holds.
\end{proof}

\begin{remark}
Note that the proof of Lemma \ref{lem:monotonicity_of_A_ker_pi} does not trivially follow by monotonicity of $\A$ in Definition \ref{def:acceptanceset}(iii), since the given position $X\in \A+\ker \pi$ does not have to be an element of the acceptance set $\A$.  
\end{remark}

For later proofs, the following corollary applies Lemma \ref{lem:monotonicity_of_A_ker_pi} for elements that can be reached from a given position in $\A+\ker \pi$ in direction of an element $U\in \M\cap \X_+$.

\begin{corollary}\label{cor:monotonicity_of_A_ker_pi_for_mU}
Consider \emph{(FM)}. Let $X\in \A+\ker \pi$. Then, 
\begin{align*}
\forall m\in \R_+, \forall U\in \M\cap \X_+:\quad X+mU\in \A+\ker \pi.
\end{align*}
More precisely,
\begin{align*}
\A+\ker \pi + \R_+U = \A + \ker \pi.
\end{align*} 
\end{corollary}

\begin{proof}
Let $U\in \M\cap \X_+$ and $m\in \R_+$. Then, the assertions follow directly from Lemma \ref{lem:monotonicity_of_A_ker_pi} by $(X+mU)-X = mU\in \X_+$ and $\A+\R_+U = \A$ by \eqref{eq:A_plus_mU_in_A}.
\end{proof}

\begin{remark}\label{rem:A_plus_ker_pi_acceptance_set}
Since $0\in \A+\ker \pi$ and Corollary \ref{cor:monotonicity_of_A_ker_pi_for_mU}, $\A+\ker \pi$ is an acceptance set itself if it is proper.  
\end{remark}

%

In the following, it will be crucial to consider the recession cone of $\A$ or $\A+\ker \pi$, respectively. 

\begin{lemma}\label{lem:rec_A}
Consider \emph{(FM)}. Then, for every $V\in \X$, it holds:
\begin{align*}
V \in \rec (\A) \qquad \Longrightarrow \qquad V\in \rec(\A+\ker \pi).
\end{align*}
\end{lemma}

\begin{proof}
Let $V\in \rec(\A)$. Take $X\in \A$ arbitrary. Then, $X=X^0+Z^0$ for some $X^0\in \A$ and $Z^0\in \ker \pi$. As a result,
\begin{align*}
X+\lambda V = (\underbrace{X^0+\lambda V}_{\in \A}) + Z^0 \in \A + \ker \pi
\end{align*}
for all $\lambda \in \R_+$ because of $V\in \rec(\A)$. Thus, $V\in \rec(\A+\ker \pi)$.
\end{proof}

As noted in Remark \ref{rem:acceptance_sets}, $U\in \M\cap\X_+$ fulfilling Assumption \ref{ann:u_payoff} satisfies $U\in \rec (\A)$. Some properties will depend on whether $-U$ belongs to the corresponding recession cone of $\A$ or $\A+\ker \pi$, respectively. We collect the previous results that are important for the proof of our main results in Theorem \ref{thm:rho_levelsets} in the following corollary: 

\begin{corollary}\label{cor:U_rec_A}
Consider \emph{(FM)}. Then, for every $U\in \M\cap \X_+$, it holds 
\begin{align*}
U\in \rec (\A) \quad \text{ and } \quad U\in \rec (\A+\ker \pi).
\end{align*}
Furthermore,
\begin{align*}
-U\in \rec (\A) \qquad \Longrightarrow \qquad -U\in \rec(\A+\ker \pi).
\end{align*}
\end{corollary}

\begin{proof}
Let $U\in \M\cap \X_+$. Take $X\in \A$ arbitrary. Then, $X+\lambda U \in \A$ for all $\lambda \in \R_+$, since $\lambda U \in \X_+$ and $\A+\X_+ \subseteq \A$ by Definition \ref{def:acceptanceset}(iii). Consequently, $U\in \rec(\A)$.  The rest follows from Lemma \ref{lem:rec_A}.
\end{proof}

In Corollary \ref{cor:U_rec_A}, the converse direction does not hold in general, although $U$ and $-U$ do not belong to $\ker \pi$, see the following example.

\begin{example}\label{expl:minus_U_in_rec_A_plus_ker_pi_not_A}
Let $\X=\M=\R^2$, $\pi\colon \M\rightarrow \R$ with 
\begin{align*}
\pi(Z)=\pi(Z_1,Z_2)=\frac{1}{2}(Z_1+Z_2),
\end{align*}
i.e., $\ker \pi = \{Z\in \R^2 \mid Z_2=-Z_1\}$. Furthermore, let $U=(1,1)^T$ and $\A\subseteq \R^2$ be an acceptance set with 
\begin{align*}
\A=\{X\in\R^2\mid X_1\geq 0\}.  
\end{align*}
Then, $\A+\ker \pi = \R^2$ and, thus, $-U\in \rec(\A+\ker \pi)$. Obviously, we have $-U\notin \rec(\A)$, since $-U=0-U\notin \A$ although $0\in \A$.
\end{example}

Next, we give some information about the domain and level sets of $\rho_{\A,\M,\pi}$, which generalize results by  Baes et al. in \cite{Baes.2020}, Lemma 2.12 (compare also Farkas et al. \cite{Farkas.2015}) where $\A$ is supposed to be closed.

\begin{theorem} \label{thm:rho_levelsets}
Consider \emph{(FM)}. Let Assumption \ref{ann:u_payoff} be fulfilled by the payoff $U\in \M\cap \X_+$. Consider the functional $\rho_{\A,\M,\pi}\colon \X\rightarrow \overline{\R}$ introduced in \eqref{def:rho} and let $m\in \R$ be arbitrary. Then, the following conditions hold:
\begin{align*}
\mathrm{(i)}&\ \lev_{\rho_{\A,\M,\pi},<}(m)=\Int_{-U}(\A+\ker \pi)-mU=\A+\ker \pi+\R_>U-mU\\ 
&\hspace{2.65 cm}\subseteq \A+\ker \pi - mU,\\
\mathrm{(ii)}&\ \lev_{\rho_{\A,\M,\pi},\leq}(m)=\cl_{-U}(\A+\ker \pi)-mU \\
 & \hspace{2.65 cm} =\cl_{-U}(\A+\ker \pi)+\R_+U - mU,\\
\mathrm{(iii)}&\ \lev_{\rho_{\A,\M,\pi},=}(m)=\bd_{-U}(\A+\ker \pi)-mU.
\end{align*}
Furthermore,
\begin{align*}
\epi \rho_{\A,\M,\pi} = \{(X,m)\in \X \times \R \mid X \in \cl_{-U}(\A+\ker \pi)-mU\}.
\end{align*}
\end{theorem}

\begin{proof}  Let $m\in\R$ be arbitrary.
\begin{enumerate}
\item By Lemma \ref{lem:k-directionally_properties_K_in_minusrec} (ii) and $U\in \rec(\A)$ by Corollary \ref{cor:U_rec_A}, we have 
\begin{align*}
\Int_{-U}(\A+\ker \pi) = \A+\ker \pi + \R_> U,
\end{align*}
showing
\begin{align}\label{eq:lev_less_int}
\Int_{-U}(\A+\ker \pi)-mU = \A+\ker \pi + \R_> U-mU.
\end{align}
By the Reduction Lemma \ref{lem:reductionlemma}, we obtain
\begin{align*}
\lev_{\rho_{\A,\M,\pi},<}(m)&=\{X\in \X \mid \rho_{\A,\M,\pi}(X)<m\}\\ 
&=\{X\in \X \mid \exists t\in \R_>: X+(m-t)U\in \A+\ker \pi\}\\
&=\{X\in \X \mid \exists t\in \R_>: X\in \A+\ker \pi - mU + tU\} \\ 
&=\A+\ker \pi +\R_> U - mU\\
&\subseteq \A+\ker \pi-mU
\end{align*}
because $\A+\R_>U\subseteq \A$ by $U\in \rec(\A)$. By \eqref{eq:lev_less_int}, we have also
\begin{align*}
\lev_{\rho_{\A,\M,\pi},<}(m)=\Int_{-U}(\A+\ker \pi)-mU.
\end{align*}
\item First, we show the second equation: By Lemma \ref{lem:k-directionally_properties}(iii), we have
\begin{align*}
\cl_{-U}(\A+\ker \pi)+\R_+ U =  \cl_{-U}(\A+\ker \pi +\R_+ U).
\end{align*}
Since $\A+\ker \pi + \R_+U = \A+\ker \pi$ by Corollary \ref{cor:monotonicity_of_A_ker_pi_for_mU}, we obtain
\begin{align*}
\cl_{-U}(\A+\ker \pi)+\R_+ U = \cl_{-U}(\A+\ker \pi).
\end{align*}
Consequently, the second equation in (ii) holds, i.e., 
\begin{align}\label{eq:lev_leq_sec_equ}
\cl_{-U}(\A+\ker \pi)+\R_+ U-mU = \cl_{-U}(\A+\ker \pi)-mU.
\end{align}
It remains to show
\begin{align}\label{eq:lev_leq}
\lev_{\rho_{\A,\M,\pi},\leq}(m)=\cl_{-U}(\A+\ker \pi)-mU. 
\end{align}
We proof ($\subseteq$) in \eqref{eq:lev_leq}: Take $X\in \X$ with $\rho_{\A,\M,\pi}(X)=m$. By the Definition of $\rho_{\A,\M,\pi}$ as an infimum and Reduction Lemma \ref{lem:reductionlemma}, it exists $(m_n)\subseteq \R$ with $m_n\downarrow m$ for $n\rightarrow +\infty$ such that 
\begin{align*}
X\in \A+\ker \pi -m_n U=\A+\ker \pi - (m_n-m)U-mU, 
\end{align*}
i.e., 
\begin{align*}
X+mU-(m_n-m)(-U)\in \A+\ker \pi. 
\end{align*}
Since $(m_n-m)\downarrow 0$, we get $X+mU\in \cl_{-U}(\A+\ker \pi)$ by Lemma \ref{lem:k-directionally_cl_sequences}. Thus, we have shown
\begin{align}\label{eq:lev_equ_subseteq}
\lev_{\rho_{\A,\M,\pi},=}(m)\subseteq \cl_{-U}(\A+\ker \pi)-mU.
\end{align}
On the other hand, we obtain by (i)
\begin{align*}
\lev_{\rho_{\A,\M,\pi},<}(m) = \A+\ker \pi + \R_> U - mU 
\end{align*}
and, thus,
\begin{align}\label{eq:lev_less_subseteq}
\lev_{\rho_{\A,\M,\pi},<}(m) \subseteq \cl_{-U}(\A+\ker \pi)+\R_+ U - mU
\end{align}
because of Lemma \ref{lem:k-directionally_properties}(i). Consequently, \eqref{eq:lev_equ_subseteq} and \eqref{eq:lev_less_subseteq} imply together
\begin{align*}
\lev_{\rho_{\A,\M,\pi},\leq}(m)\subseteq \cl_{-U}(\A+\ker \pi)+\R_+ U - mU
\end{align*}
because $0\in \R_+U$ which leads by \eqref{eq:lev_leq_sec_equ} to
\begin{align*}
\lev_{\rho_{\A,\M,\pi},\leq}(m)\subseteq \cl_{-U}(\A+\ker \pi) - mU,
\end{align*}
showing ($\subseteq$) in \eqref{eq:lev_leq}.

Now, we proof ($\supseteq$) in \eqref{eq:lev_leq}: 
By Lemma \ref{lem:k-directionally_properties}(iv), we have
\begin{align*}
\cl_{-U}(\A+\ker \pi)+\R_>U&= \Int_{-U}(\A+\ker \pi +\R_+ U)
\end{align*}
and, thus, 
\begin{align*}
\cl_{-U}(\A+\ker \pi)+\R_>U&=\Int_{-U}(\A+\ker \pi)
\end{align*}
by \eqref{eq:A_plus_mU_in_A}. Therefore,
\begin{align*}
\cl_{-U}(\A+\ker \pi)+\R_>U-mU = \lev_{\rho_{\A,\M,\pi},<}(m) \subseteq  \lev_{\rho_{\A,\M,\pi},\leq}(m)
\end{align*}  
by (i). This shows ($\supseteq$) in \eqref{eq:lev_leq}.
\item The assertion follows by (i) and (ii) through
\begin{align*}
\lev_{\rho_{\A,\M,\pi},=}(m)&=\lev_{\rho_{\A,\M,\pi},\leq}(m)\backslash \lev_{\rho_{\A,\M,\pi},<}(m)\\ 
&= (\cl_{-U}(\A+\ker \pi)-mU)\backslash (\Int_{-U}(\A+\ker \pi)-mU)\\
&= (\cl_{-U}(\A+\ker \pi)\backslash \Int_{-U}(\A+\ker \pi))-mU\\
&=\bd_{-U}(\A+\ker \pi)-mU.
\end{align*}
\end{enumerate}
The description of $\epi \rho_{\A,\M,\pi}$ follows from Theorem \ref{thm:rho_levelsets}(ii).
\end{proof}
\newpage

\begin{corollary} \label{cor:rho_levelsets}
Consider \emph{(FM)}. Let Assumption \ref{ann:u_payoff} be fulfilled by the payoff $U\in \M\cap \X_+$. Consider the functional $\rho_{\A,\M,\pi}\colon \X \rightarrow \overline{\R}$ given by \eqref{def:rho} and let $m\in \R$ be arbitrary. Then,  the following holds:
\begin{enumerate}[font = \normalfont]
\item $\lev_{\rho_{\A,\M,\pi},\leq}(m) \supseteq \A+\ker \pi -mU$,
\item $\lev_{\rho_{\A,\M,\pi},\leq}(m)=\A+\ker \pi-mU$ holds if and only if $\A+\ker \pi$ is $(-U)$-directionally closed.
\end{enumerate}
\end{corollary}

\begin{proof} \mbox{}\vspace{\topskip}
 
\begin{enumerate}
\item Since 
\begin{align*}
\A+\ker \pi - mU \subseteq \cl_{-U}(\A+\ker \pi)-mU
\end{align*}
by Lemma \ref{lem:k-directionally_properties}(i), the assertion follows from Theorem \ref{thm:rho_levelsets}(ii).
\item follows by Tammer and Weidner \cite[Prop. 4.2.1(b)]{Tammer.2020}.
%
\end{enumerate}
\end{proof}

Under the assumption of $\rho_{\A,\M,\pi}$ being continuous and finite on $\X$, Baes et al. studied the sets $\lev_{\rho_{\A,\M,\pi},<}(m)$, $\lev_{\rho_{\A,\M,\pi},\leq}(m)$ and $\lev_{\rho_{\A,\M,\pi},=}(m)$ for the special case $m=0$ and $\A$ being a closed acceptance set in \cite[Lemma 2.12]{Baes.2020} (compare also Farkas et al. \cite{Farkas.2015}). They observed the following:

\begin{lemma}[see Baes et al. \cite{Baes.2020}, Lemma 2.12]\label{lem:level_sets_rho_baes}
Consider \emph{(FM)}. Let $\X$ be a locally convex Hausdorff topological vector space  over $\R$ fulfilling the first axiom of countability. Furthermore, let $\A$ be a closed acceptance set and Assumption \ref{ann:u_payoff} be fulfilled. Consider the functional $\rho_{\A,\M,\pi}\colon \X \rightarrow \overline{\R}$ given by \eqref{def:rho}. Suppose that $\rho_{\A,\M,\pi}$ is continuous and finite on $\X$. Then, the following conditions hold:
\begin{enumerate}[font = \normalfont]
\item $\lev_{\rho_{\A,\M,\pi},<}(0)=\Int(\A+\ker \pi)$,
\item $\lev_{\rho_{\A,\M,\pi},\leq}(0)=\cl(\A+\ker \pi)$,
\item $\lev_{\rho_{\A,\M,\pi},=}(0)=\bd(\A+\ker \pi)$.
\end{enumerate}
\end{lemma}

\bigskip
These properties look similar to our results for $m=0$, but we derived $\Int_{-U}(\A+\ker \pi)$, $\cl_{-U}(\A+\ker \pi)$ and $\bd_{-U}(\A+\ker \pi)$ in Theorem \ref{thm:rho_levelsets}. This can be united with our results by the following theorem that is formulated for normed vector spaces in order to use sequences:

\begin{theorem}\label{thm:directional_bd_int_cl_coincides_with_non_directional}
Consider \emph{(FM)}. Let $(\X,\norm{\cdot})$ be a normed vector space over $\R$ and Assumption \ref{ann:u_payoff} be fulfilled by the payoff $U\in \M\cap \X_+$. Consider the functional $\rho_{\A,\M,\pi}\colon \X \rightarrow \overline{\R}$ given by \eqref{def:rho}. Suppose that one of the following conditions is satisfied:
\begin{enumerate}[label = (\alph*) , font = \normalfont]
\item $\rho_{\A,\M,\pi}$ is continuous on $\X$,
\item $\A+\ker \pi$ fulfills
\begin{align}\label{eq:directional_prop_Int_A}
\A+\ker \pi+\R_>U\subseteq \Int(\A+\ker \pi)
\end{align}
and
\begin{align}\label{eq:directional_prop_cl_A}
\cl(\A+\ker \pi)+\R_>U \subseteq \A+\ker \pi.
\end{align}
\end{enumerate}
Then, the following conditions hold:
\begin{enumerate}[font = \normalfont]
\item $\Int_{-U}(\A+\ker \pi) = \Int (\A + \ker \pi),$ 
\item $\cl_{-U}(\A+\ker \pi) = \cl (\A + \ker \pi), $
\item $\bd_{-U}(\A+\ker \pi) = \bd (\A+\ker \pi)$.
\end{enumerate}
\end{theorem}

\begin{proof}
Suppose that (a) is fulfilled. 
\begin{enumerate}
\item The relation ($\supseteq$) follows by Lemma \ref{lem:k-directionally_topol_properties}(ii). Thus, we need to show ($\subseteq$): Consider $X\in \Int_{-U}(\A+\ker \pi)$. Then, 
\begin{align*}
(X-\R_> U)\cap(\A+\ker \pi)\neq \emptyset
\end{align*}
and, thus, $\rho_{\A,\M,\pi}(X)<0$. Suppose, for every $n\in \N$ exists $X_n\in \X$ with 
\begin{align*}
X_n\in \B_\frac{1}{n}(X) \text{ and }X_n\notin \A+\ker \pi. 
\end{align*}
Then, $(X_n)_{n\in \N}\subseteq \X$ is a sequence with $X_n\rightarrow X$ for $n\rightarrow +\infty$ $\mathbb{P}$-a.s. and 
\begin{align*}
\rho_{\A,\M,\pi}(X_n)\rightarrow \rho_{\A,\M,\pi}(X)<0 \text{ for } n\rightarrow +\infty
\end{align*}
because of the continuity of $\rho_{\A,\M,\pi}$. By convergence, for every $\delta > 0$ there is some $N(\delta)\in \N$ with 
\begin{align*}
\abs{\rho_{\A,\M,\pi}(X_k)-\rho_{\A,\M,\pi}(X)}<\delta \text{ for all }k>N(\delta). 
\end{align*}
By choice of $\delta=\abs{\frac{\rho_{\A,\M,\pi}(X)}{2}}$, we get $\rho_{\A,\M,\pi}(X_k)<0$ and, thus, 
\begin{align*}
\qquad X_k\in \Int_{-U}(\A+\ker \pi)\subseteq \A+\ker \pi \text{ for }k>N\left(\delta\right)
\end{align*}
by Theorem \ref{thm:rho_levelsets}(i) and Lemma \ref{lem:k-directionally_topol_properties}(ii).  Consequently, such an sequence $(X_n)_{n\in\N}$ cannot exist and there is some $n_0\in \N$ with $\B_\frac{1}{n_0}(X)\subseteq \A+\ker \pi$. Thus, we obtain $X\in \Int(\A+\ker \pi)$, which shows $(\subseteq)$ and, therefore, (i).    
\item The relation ($\subseteq$) follows by Lemma \ref{lem:k-directionally_topol_properties}(i). We just need to show ($\supseteq$) in (ii): Consider now $X\in \cl(\A+\ker \pi)$. Then, 
\begin{align*}
\exists (X_n)_{n\in\N}\subseteq \A+\ker \pi: \quad X_n\rightarrow X \text{ for } n\rightarrow +\infty\ \mathbb{P}-a.s. 
\end{align*}
Thus, $\rho_{\A,\M,\pi}(X_n)\leq 0$ by the Reduction Lemma \ref{lem:reductionlemma}, which implies $\rho_{\A,\M,\pi}(X)\leq 0$ by continuity of $\rho_{\A,\M,\pi}(X)$. Consequently, we have $X\in \cl_{-U}(\A+\ker \pi)$ by Theorem \ref{thm:rho_levelsets}(ii). 
\item The assertion follows with (i) and (ii) by 
\begin{align*}
\bd_{-U}(\A+\ker \pi)&=\cl_{-U}(\A+\ker \pi)\backslash \Int_{-U}(\A+\ker \pi)\\ 
&=\cl(\A+\ker \pi)\backslash \Int(\A+\ker \pi)\\ 
&=\bd(\A+\ker \pi).
\end{align*}
\end{enumerate}
Suppose now that (b) holds. 
\begin{enumerate}
\item As in (a), we have to show ($\subseteq$) in (i): Let $X\in \Int_{-U}(\A+\ker \pi)$. Then, it exists $m\in \R_>$ with $X-mU\in \A+\ker \pi$. Thus, 
\begin{align*}
X\in \A+\ker \pi +\R_> U\subseteq \Int(\A+\ker \pi)
\end{align*}
by \eqref{eq:directional_prop_Int_A}, showing (i). 
\item Again, ($\subseteq$) follows by Lemma \ref{lem:k-directionally_topol_properties}(i). The relation ($\supseteq$) is a direct consequence of \eqref{eq:directional_prop_cl_A} by Lemma \ref{lem:k-directionally_properties_K_in_minusrec}(i).
\item The assertion follows like in (a) by (i) and (ii). 
\end{enumerate}
\end{proof}

\begin{remark}
In more detail, since $U\in \rec(\A+\ker \pi)$ for $U\in \M\cap \X_+$ (see Corollary \ref{cor:U_rec_A}), it can be shown that the following holds: 
\begin{align*}
&\eqref{eq:directional_prop_Int_A}\quad  \Longleftrightarrow\quad \Int_{-U}(\A+\ker \pi)=\Int(\A+\ker \pi),\\ 
&\eqref{eq:directional_prop_cl_A}\quad \Longleftrightarrow\quad \cl_{-U}(\A+\ker \pi)=\cl(\A+\ker \pi),
\end{align*}
see also \cite{Tammer.2020}, Prop. 2.3.54 and Prop. 2.3.55. The subspace $\ker \pi$ is closed even if  $\dim \X = +\infty$, since $\pi$ is linear and continuous, but, even if $\A$ is closed and we have a sum of two closed sets, the augmented set $\A+\ker \pi$ does not have to be closed or ($-U$)-directionally closed.
\end{remark}

%

Consider $\rho_{\A,\M,\pi}$ introduced in \eqref{def:rho}. Remember that the set 
\begin{align*}
\A_{\rho_{\A,\M,\pi}}:=\lev_{\rho_{\A,\M,\pi},\leq}(0)=\{X\in \X\mid \rho_{\A,\M,\pi}(X)\leq 0\}
\end{align*}
in Lemma \ref{lem:level_sets_rho_baes} (ii) is an acceptance set itself by Example \ref{expl:VaR_CVaR_acc_sets} because $\rho_{\A,\M,\pi}$ is a monetary risk measure with $\rho_{\A,\M,\pi}(0)\leq 0$. Indeed, it holds
\begin{align*}
\rho_{\A,\M,\pi} \text{ normalized } \quad \Longrightarrow \quad \rho_{\A,\M,\pi}(0)=0,
\end{align*}
which is given if $\A$ fulfills \eqref{eq:span_U_no_good_deals} and Assumption \ref{ann:absence_good_deals} holds (see Remark \ref{rem:rho_normalized}). Moreover, ${U\in \M\cap \X_+}$ from Assumption \ref{ann:u_payoff} fulfills $U\in \A_{\rho_{\A,\M,\pi}}$ and $U\in \rec (\A_{\rho_{\A,\M,\pi}})$, as well, see Remark \ref{rem:acceptance_sets}. 

\bigskip 
The relationship between the acceptance set $\A_{\rho_{\A,\M,\pi}}$ and the translation invariant functional $\rho_{\A,\M,\pi}$ introduced in \eqref{def:rho} can be described as followed:

\begin{theorem}
Consider \emph{(FM)}. Let Assumption \ref{ann:u_payoff} be fulfilled by $U\in \M\cap \X_+$. Consider the functional $\rho_{\A,\M,\pi}\colon \X \rightarrow \overline{\R}$ given by \eqref{def:rho}. Furthermore, let $\A_{\rho_{\A,\M,\pi}}:=\lev_{\rho_{\A,\M,\pi},\leq}(0)$.  Then, $\A_{\rho_{\A,\M,\pi}}$ fulfills
\begin{align}\label{eq:unique_level_set}
\lev_{\rho_{\A,\M,\pi},\leq}(m)=\A_{\rho_{\A,\M,\pi}} - mU \text{ for all }m\in \R.
\end{align}
Moreover, $\A_{\rho_{\A,\M,\pi}}$ has the following properties:
\begin{enumerate}[font = \normalfont]
\item $\A_{\rho_{\A,\M,\pi}}$ is ($-U$)-directionally closed,
\item $\rho_{\A,\M,\pi} \equiv \rho_{\A_{\rho_{\A,\M,\pi}},\M,\pi}$ on $\X$.
\end{enumerate}
\end{theorem}

\begin{proof}
From Theorem \ref{thm:rho_levelsets}(ii), we obtain for $m=0$ 
\begin{align*}
\A_{\rho_{\A,\M,\pi}} = \cl_{-U}(\A+\ker \pi)+\R_+ U. 
\end{align*}
This yields
\begin{align*}
\A_{\rho_{\A,\M,\pi}}-mU = \cl_{-U}(\A+\ker \pi)+\R_+ U-mU = \lev_{\rho_{\A,\M,\pi},\leq}(m)
\end{align*}
for all $m\in \R$ , showing \eqref{eq:unique_level_set}.
As a result, we obtain by Theorem \ref{thm:rho_levelsets}(ii) for $m=0$
\begin{align*}
\A_{\rho_{\A,\M,\pi}} = \lev_{\rho_{\A,\M,\pi},\leq}(0)=\cl_{-U}(\A + \ker \pi)
\end{align*} 
which is obviously a  $(-U)$-directionally closed set, i.e., (i) holds.

In order to proof (ii), we show
\begin{align}\label{eq:A_rho_plus_ker_pi_equals}
\A_{\rho_{\A,\M,\pi}} + \ker \pi &= \A_{\rho_{\A,\M,\pi}}.
\end{align}
Since $0\in \ker \pi$, ($\supseteq$) in \eqref{eq:A_rho_plus_ker_pi_equals} is clear. For ($\subseteq$), consider $X\in \A_{\rho_{\A,\M,\pi}} + \ker \pi$ arbitrary. Then, there is some $Z^0\in \ker \pi$ with
\begin{align*}
X^0:=X+Z^0\in \A_{\rho_{\A,\M,\pi}}= \cl_{-U}(\A+\ker \pi).
\end{align*}
Consequently,
\begin{align*}
\forall \lambda \in \R_>:\quad X^0+\lambda U \in \A+\ker \pi.
\end{align*}
Since $-Z^0\in \ker \pi$, we have $X^0+\lambda U - Z^0\in \A +\ker \pi $ for every $\lambda \in \R_>$. Hence, we obtain
\begin{align*}
\forall n\in \N:\quad X^0+\frac{1}{n}U-Z^0 = X+\frac{1}{n}U\in \A+\ker \pi
\end{align*}
and, thus,
$X\in \cl_{-U}(\A+\ker \pi)$ by Lemma \ref{lem:k-directionally_cl_sequences}, showing ($\subseteq$) in \eqref{eq:A_rho_plus_ker_pi_equals}. That completes the proof of \eqref{eq:A_rho_plus_ker_pi_equals}.

Formula \eqref{eq:unique_level_set} and \eqref{eq:A_rho_plus_ker_pi_equals} imply
\begin{align*}
\rho_{\A,\M,\pi}(X)&=\inf\{m\in \R \mid X \in \lev_{\rho_{\A,\M,\pi},\leq}(m)\}\\ 
&=\inf\{m\in \R \mid X \in \A_{\rho_{\A,\M,\pi}}-mU\}\\ 
&=\inf\{m\in \R \mid X+mU \in \A_{\rho_{\A,\M,\pi}}\}\\ 
&=\inf\{m\in \R \mid X+mU \in \A_{\rho_{\A,\M,\pi}} + \ker \pi\}\\ 
&=\rho_{\A_{\rho_{\A,\M,\pi}},\M,\pi}(X)
\end{align*}
for all $X\in \X$. Here, the last equation follows by the Reduction Lemma \ref{lem:reductionlemma}. Hence, (ii) holds.

\end{proof}

We can vary the acceptance set $\A$ in some range without changing the values of $\rho_{\A,\M,\pi}$, which is stated in the following lemma:

\begin{lemma}
Consider \emph{(FM)}. Let Assumption \ref{ann:u_payoff} be fulfilled by $U\in \M\cap \X_+$. Consider the functional $\rho_{\A,\M,\pi}\colon \X \rightarrow \overline{\R}$ given by \eqref{def:rho}. Then, $\rho_{\A,\M,\pi}=\rho_{\D,\M,\pi}$ for every set $\D\subseteq \X$ with 
\begin{align}\label{eq:D_as_variation_of_A}
\A+\ker \pi \subseteq \D+\ker \pi \subseteq \cl_{-U}(\A+\ker \pi). 
\end{align}
\end{lemma}

\begin{proof}
 Consider $\D \subseteq \X$ fulfilling \eqref{eq:D_as_variation_of_A}. Then,
\begin{align*}
\cl_{-U}(\A+\ker \pi) \subseteq \cl_{-U}(\D+\ker \pi).
\end{align*}
Furthermore, we get 
\begin{align*}
\cl_{-U}(\D+\ker \pi) \subseteq \cl_{-U}(\cl_{-U}(\A+\ker \pi)) = \cl_{-U}(\A+\ker \pi)
\end{align*}
by \eqref{eq:D_as_variation_of_A} and Lemma \ref{lem:k-directionally_properties}(ii). Thus, 
\begin{align*}
\cl_{-U}(\A+\ker \pi) = \cl_{-U}(\D+\ker \pi).
\end{align*}
This yields
\begin{align*}
\cl_{-U}(\A+\ker \pi)-mU = \cl_{-U}(\D+\ker \pi)-mU
\end{align*}
for all $m\in \R$, i.e., $\lev_{\rho_{\A,\M,\pi},\leq}(m)=\lev_{\rho_{\D,\M,\pi},\leq}(m)$ by Theorem \ref{thm:rho_levelsets}(ii). As a result, we obtain $\rho_{\A,\M,\pi}=\rho_{\D,\M,\pi}$.
\end{proof}

As for every risk measure we are interested in finiteness of $\rho_{\A,\M,\pi}$. 

\begin{theorem}\label{thm:rho_set_of_finite_values}
Consider \emph{(FM)}. Let Assumption \ref{ann:u_payoff} be fulfilled by $U\in \M\cap \X_+$. Consider the functional $\rho_{\A,\M,\pi}\colon \X \rightarrow \overline{\R}$ given by \eqref{def:rho}. Then,
\begin{align}\label{eq:dom_rho}
\dom \rho_{\A,\M,\pi} = \A +\ker \pi + \R U = \A + \M.
\end{align}
Moreover, take $X\in \X$ arbitrary. Then,
\begin{align}\label{eq:rho_finite}
\rho_{\A,\M,\pi}(X)\in \R \qquad \Longleftrightarrow \qquad X\in\bd_{-U}(\A+\ker \pi)+\R U.
\end{align}
\end{theorem}

\begin{proof}
 The equivalence \eqref{eq:rho_finite} follows directly from Theorem \ref{thm:rho_levelsets}(iii). 
 
 So, we only have to show \eqref{eq:dom_rho}. We start with the first equation: Let $X\in\dom \rho_{\A,\M,\pi}$. By Reduction Lemma \ref{lem:reductionlemma}, there exists $m\in \R$ with $X+mU\in \A+\ker \pi$ or, equivalently, ${X\in \A+\ker \pi - mU}$. Thus, $X\in \A+\ker \pi + \R U$, showing
\begin{align*}
\dom \rho_{\A,\M,\pi} \subseteq \A +\ker \pi + \R U
\end{align*}
in \eqref{eq:dom_rho}. Conversely, let $X\in \A +\ker \pi + \R U$. Then, it exists $m\in \R$ with
\begin{align*}
X+ mU \in \A+ \ker \pi,
\end{align*}
which yields $\rho_{\A,\M,\pi} (X)\leq m < +\infty$ by Reduction Lemma \ref{lem:reductionlemma} and, thus, $X\in \dom \rho_{\A,\M,\pi}$. That completes the proof of the first equation in \eqref{eq:dom_rho}. 

In the second equation, 
\begin{align*}
\M \supseteq \ker \pi + \R U
\end{align*}
holds obviously, since $U\in \M$ and $\M$ being a subspace of $\X$. Now, let $Z\in \M$ arbitrary. Then, $\pi(Z)U\in \M$ and $Z-\pi(Z)U\in \M$. By linearity of $\pi$, 
\begin{align*}
\pi(Z-\pi(Z)U)=\pi(Z)-\pi(\pi(Z)U)=\pi(Z)-\pi(Z)\cdot 1=0,
\end{align*}
since $\pi(U)=1$ by Assumption \ref{ann:u_payoff}. Consequently, 
\begin{align*}
Z=(Z-\pi(Z)U)+\pi(Z)U \in \ker \pi + \R U,
\end{align*}
showing $\M\subseteq \ker \pi + \R U$, which complets the proof of \eqref{eq:dom_rho} by
\begin{align*}
\A+\ker \pi +\R U = \A + \M.
\end{align*}
\end{proof}

Farkas et al. \cite{Farkas.2015} observed the following lemma for topological vector spaces $\X$, which is proved in our paper \cite[Lemma 3.16]{Marohn.2020} without using any topological properties:
\begin{lemma}[see \cite{Farkas.2015}] \label{lem3-16}
Consider \emph{(FM)}. Let Assumption \ref{ann:u_payoff} be fulfilled by $U\in \M\cap \X_+$. Consider the functional $\rho_{\A,\M,\pi}\colon \X \rightarrow \overline{\R}$ given by \eqref{def:rho}. Furthermore, let $\A+\ker \pi = \X$. Then,  $\rho_{\A,\M,\pi}\equiv -\infty$.
\end{lemma}

We can also give a condition as in Lemma \ref{lem3-16} restricted to capital positions in the domain of $\rho_{\A,\M,\pi}$, which is related to the payoff $U$:

\begin{lemma}\label{lem:rho_infinite_by_U}
Consider \emph{(FM)}. Let Assumption \ref{ann:u_payoff} be fulfilled by $U\in \M\cap \X_+$. Consider the functional $\rho_{\A,\M,\pi}\colon \X \rightarrow \overline{\R}$ given by \eqref{def:rho}. Furthermore, let $-U\in \rec (\A)$ hold. Then,  $\rho_{\A,\M,\pi}(X) = -\infty$ for all $X\in \dom \rho_{\A,\M,\pi}$.
\end{lemma}

\begin{proof}
Let $X\in \dom \rho_{\A,\M,\pi}$ arbitrary. By the Reduction Lemma \ref{lem:reductionlemma}, it exists $m\in \R$ with $X+mU\in \A+\ker \pi$ or, equivalently, $X\in \A+\ker \pi - mU$. Since $-U\in \rec (\A)$, we have 
\begin{align*}
X\in \A+\ker \pi - tU \text{ for every } t\leq m
\end{align*}
and, thus, $\rho_{\A,\M,\pi}(X)\leq t$ for all $t\leq m$, which shows $\rho_{\A,\M,\pi}(X)=-\infty$. 
\end{proof}

\begin{remark}
Under our assumptions, $U\in \rec (\A)$ holds, see Remark \ref{rem:acceptance_sets}. So, we suppose in Lemma \ref{lem:rho_infinite_by_U} $U\in \rec (\A)$  and $-U\in \rec (\A)$.
\end{remark}

\begin{remark}
Note that $-U\in \rec(\A)$ is not necessary for $\A+\ker \pi=\X$, although it holds $U\in \rec(\A)$ for $U\in\M\cap\X_+$ arbitrary. For example, let $\X=\M=\R^2$, $\pi(Z_1,Z_2)=\frac{Z_1+Z_2}{2}$ and
\begin{align*}
\A_1&=\left\{(X_1,X_2)^T\in \X \mid X_1\leq 0, X_2\geq \sqrt{-X_1}\right\} \cup \left\{(X_1,X_2)^T\in \X \mid X_1>0, X_2=0\right\}.
\end{align*}
Then, $-U\notin \rec(\A_1)$ for every $U\in \M\cap \X_+$, but $\A_1+\ker \pi = \X$.

In general, $-U\in \rec(\A)$ is not sufficient for $\A+\ker \pi=\X$, as well. Consider $\X=\R^3$ and $\M=\{0\}\times\R\times\R$. Let $U=(0,0,1)^T$ and 
\begin{align*}
\A_2=\R U + \R^3_+.
\end{align*}
Then, $-U\in \rec(\A_2)$ and $U\in \rec(\A_2)$. Consider $\pi\colon \M\rightarrow \R$ with $\pi(Z)=\pi(Z_1,Z_2,Z_3):=Z_3$. Then,
\begin{align*}
\A_2+\ker \pi = \A_2 + \R (0,1,0)^T = \{(X_1,X_2,X_3)^T\in \R^3\mid X_1\geq 0\},
\end{align*}
i.e., $\A_2+\ker \pi \neq \R^3=\X$. However, $-U\in \rec(\A)$ is sufficient for $\A+\ker \pi = \X$ if we require  $n:=\dim(\X)<+\infty$ and $\dim(\ker \pi)=n-1$. The reason is that the subspaces $\ker \pi$ and $\R U$ of $\M\subseteq \X$ fulfill $\ker \pi \cap \R U = \{0\}$ and, thus, for their direct sum $\dim(\ker \pi + \R U) = n$ holds.

In general, Lemma \ref{lem:rho_infinite_by_U} secures finiteness only for $X\in \dom \rho_{\A,\M,\pi}$ instead of the whole space $\X$. 
\end{remark}

\begin{remark}
Under $\A + \ker \pi \neq \X$, it is impossible to make every capital position acceptable by zero costs, which is also called \textit{absense of acceptability arbitrage}, see Artzner et al. in \cite{Artzner.2009}. In topological vector spaces $\X$ Baes et al. \cite[Prop. 2.10]{Baes.2020} and Farkas et al. \cite{Farkas.2015} observe different sufficient conditions for $\rho_{\A,\M,\pi}$ being finite and continuous if $\A+\ker \pi \neq \X$ is fulfilled, e.g. $\Int \X_+\cap \M \neq \emptyset$. Also, we refer to \cite[Section 3]{Farkas.2015} for specific conditions for finiteness of $\rho_{\A,\M,\pi}$ under certain properties of the acceptance set like being convex or coherent.
\end{remark}

The observations for $\rho_{\A,\M,\pi}\equiv -\infty$ lead to the following equivalence, which gives more details:

\begin{theorem}
Consider \emph{(FM)}. Let Assumption \ref{ann:u_payoff} be fulfilled by $U\in \M\cap \X_+$. Consider the functional $\rho_{\A,\M,\pi}\colon \X \rightarrow \overline{\R}$ given by \eqref{def:rho}. Then, $\rho_{\A,\M,\pi}$ is proper if and only if $\A+\ker \pi$ does not contain lines parallel to $U$, i.e.,
\begin{align}\label{eq:parallel_lines_to_U}
X+\R U \not\subseteq \A+\ker \pi \qquad \text{ for all } X\in \A+\ker\pi.
\end{align}
\end{theorem}

\begin{proof}
First, we note that \eqref{eq:parallel_lines_to_U} is equivalent to
\begin{align*}
\forall X\in \A+\ker \pi \ \exists m\in \R:\ X+mU \notin \A+\ker \pi. 
\end{align*}
Let $\rho_{\A,\M,\pi}$ be proper. Then, $\rho_{\A,\M,\pi}(X)>-\infty$ for all $X\in \X$, i.e., it exists $m\in \R$ with 
\begin{align*}
X+tU\notin \A+\ker \pi \text{ for all } t<m
\end{align*}
by the Reduction Lemma \ref{lem:reductionlemma}. Hence,  \eqref{eq:parallel_lines_to_U} holds.

Conversely, let \eqref{eq:parallel_lines_to_U} be fulfilled. Thus, if $X+mU\in \A+\ker \pi$,  there is some $t\in \R_>$ with 
\begin{align*}
X+(m-t)U\notin \A + \ker \pi, 
\end{align*}
i.e., $\rho_{\A,\M,\pi}(X)>-\infty$. Furthermore, since $0\in \A$ by Definition \ref{def:acceptanceset}(i) and, thus, $0\in\A+\ker \pi$ holds, it is $\rho_{\A,\M,\pi}(0)<+\infty$. Hence, $\dom \rho_{\A,\M,\pi}\neq \emptyset$, i.e., $\rho_{\A,\M,\pi}$ is proper.  
\end{proof}

\begin{remark}
In \eqref{eq:parallel_lines_to_U}, $X\in \A+\ker \pi$ can be replaced by $X\in \X$. Note that $-U\notin \rec(\A)$ is necessary for \eqref{eq:parallel_lines_to_U} because if $-U\in \rec(\A)$ holds for $U\in \M\cap \X_+$ fulfilling Assumption \ref{ann:u_payoff}, we have \begin{align*}
X-mU\in \A\subseteq \A+\ker \pi
\end{align*}
for every $X\in \A$ and $m\in \R_+$, which contradicts \eqref{eq:parallel_lines_to_U}, since $X+mU\in \A+\ker \pi$ for every $m\in \R$, too, see Corollary \ref{cor:U_rec_A}. On the other hand, Example \ref{expl:minus_U_in_rec_A_plus_ker_pi_not_A} shows that $-U\notin \rec(\A)$ is not sufficient for \eqref{eq:parallel_lines_to_U}. 
\end{remark}

\bigskip
The following lemma from \cite[Lemma 2.8]{Farkas.2015} summarizes some more properties of $\rho_{\A,\M,\pi}$ that imply conditions under which $\rho_{\A,\M,\pi}$ is a convex or coherent risk measure (see Remark \ref{rem:rho_normalized}). The proof can be found in \cite[Lemma 3.20]{Marohn.2020}, which does not require any topological properties (compare \cite[Theorem 2.3.1]{Gopfert.2003}).

\begin{lemma}[see \cite{Farkas.2015}, Lemma 2.8]\label{lem:further_properties_of_rho}
Consider \emph{(FM)}. Let Assumption \ref{ann:u_payoff} be fulfilled. Consider the functional $\rho_{\A,\M,\pi}\colon \X \rightarrow \overline{\R}$ given by \eqref{def:rho}. Then,
$\rho_{\A,\M,\pi}$  satisfies the following properties:
\begin{enumerate}[font = \normalfont]
\item $\rho_{\A,\M,\pi}$ is convex if $\A$ is convex,
\item $\rho_{\A,\M,\pi}$ is subadditive if $\A$ is closed under addition, i.e., $X+Y\in \A$ for all $X,Y\in \A$,
\item $\rho_{\A,\M,\pi}$ is positively homogeneous if $\A$ is a cone.
\end{enumerate}
\end{lemma}

\section{Conclusion}

In our paper, we studied properties of a risk measure $\rho_{\A,\M,\pi}$ associated with a not necessary closed acceptance set $\A\subseteq \X$, a space of eligible payoffs $\M\subseteq \X$ and a pricing functional $\pi\colon \M \rightarrow \R$. The study of $\rho_{\A,\M,\pi}$ was motivated by Baes et al. in \cite{Baes.2020} where solutions of an optimization problem referring to $\rho_{\A,\M,\pi}$ were subject to the investigation. $\rho_{\A,\M,\pi}$ is a monetary risk measure and, therefore, translation invariant. As seen in the paper, it is suitable for scalarization in the frame of multiobjective optimization because it is directly connected with the functional given in \eqref{funcak00}, which plays an important role in optimization. We have shown important properties of the translation invariant functional $\rho_{\A,\M,\pi}$. Especially, we studied the properties of the sublevel sets, strict sublevel sets and level lines of $\rho_{\A,\M,\pi}$. Furthermore, we discussed the finiteness of the functional and relaxed closedness assumptions concerning $\A+\ker \pi$.

For further research, it would be interesting to consider general acceptance sets and use the properties of $\rho_{\A,\M,\pi}$ derived in this paper for studying the optimization problem of making the initial capital position acceptable with minimal costs for general acceptance sets using a scalarization approach by means of our functional $\rho_{\A,\M,\pi}$. 

\bigskip

\printbibliography 


%
%
%
%
%
%
%
%
%
%
%
%
%
%
%


\end{document}